\documentclass[final]{ws}

\usepackage{environ,tikz}
\usepackage{amsmath,amsfonts}
\usepackage{bbm,enumerate,stackrel}
\usepackage{paralist}
\usepackage[pdfauthor={Vuk Mili\v si\'c, CNRS/LAGA Université Paris 13},pdftitle={From a delayed constrained diffusion to the harmonic map heat equation},pdftex]{hyperref}
\usepackage[cyr]{aeguill}
\usepackage[a4paper, total={6in, 8in}]{geometry}
\def\ov{\overline}
\def\supp{\hbox{\rm supp}}

\def\BV {{\rm BV}}

\def\D{{\mathcal D}}

\newcommand{\RR}{\mathbb R}

\newcommand{\N}{\mathbb N}

\newcommand{\R}{\mathbb R}
\newcommand{\Z}{\mathbb Z}

\newcommand{\rr}{{\RR_+}}

\newcommand{\ue}{\frac{1}{\varepsilon}}

\newcommand{\dt}{\partial_t}

\newcommand{\dx}{\partial_x}
\newcommand{\dxx}{\partial_{xx}}
\newcommand{\e}{\varepsilon}

\newcommand{\normel}{\left\|}
\newcommand{\normer}{\right\|}

\newcommand{\nud}{{\frac{1}{2}}}

\newcommand{\td}{\frac{3}{2}}

\newcommand{\ddt}[1]{\frac{{\rm d} #1}{\rm d \it t}}
\newcommand{\ti}[1]{\tilde{#1}}

\newtheorem{hypo}{Assumptions}[section]

\newcommand{\nrm}[2]{{\normel{#1}\normer}_{#2}}

\newcommand{\SysNoNumB}[1]{
\begin{displaymath}
\left\{
\begin{array}{#1}
}
\newcommand{\SysNoNumE}{
\end{array}
\right.
\end{displaymath}
}
\newcommand{\TabNoNumB}[1]{
\begin{displaymath}
\begin{array}{#1}
}
\newcommand{\TabNoNumE}{
\end{array}
\end{displaymath}
}
\newcommand{\TabNumB}[2]{
\begin{equation}
\label{#2}
\begin{array}{#1}
}
\newcommand{\TabNumE}{
\end{array}
\end{equation}
}
\newcommand{\SysNumB}[2]{
\begin{equation}
\left\{
\label{#2}
\begin{array}{#1}
}
\newcommand{\SysNumE}{
\end{array}
\right.
\end{equation}
}

\newcommand{\Dt}{\Delta t}

\newcommand{\ddn}[1]{\partial_{\bfx} #1}
\newcommand{\cM}{\mathcal{M}}

\newcommand{\Da}{{\Delta a}}

\newcommand{\etau}[1]{}

\newcommand{\chiu}[1]{\ensuremath{\chi_{#1}}}

\newcommand{\cR}{{\mathcal R}}

\newcommand{\cH}{{\mathcal H}}

\newcommand{\cD}{{\mathcal D}}

\newcommand{\id}[1]{{\rm I}_{#1}}

\newcommand{\rhoe}{\rho_\e}

\newcommand{\rhoi}{\rho_{I}}
\newcommand{\rhoz}{\rho_0}

\newcommand{\da}{\partial_a} 
\newcommand{\zeps}{{\mathbf  z}_\e}

\newcommand{\zp}{{\mathbf  z}_p}

\newcommand{\zpD}{{\mathbf z}_{p,\Delta}}

\newcommand{\bmax}{\beta_{\max}}
\newcommand{\bmin}{\beta_{\min}}
\newcommand{\ztmax}{\zeta_{\max}}
\newcommand{\ztmin}{\zeta_{\min}}
\newcommand{\mumin}{\mu_{0,\min}}

\newcommand{\beps}{\beta_{\varepsilon}}
\newcommand{\bz}{\beta_0}
\newcommand{\zteps}{\zeta_\e}
\newcommand{\ztz}{\zeta_0}
\newcommand{\hrhoe}{\hat{\rho}_\varepsilon}

\newcommand{\tit}{\tilde{t}}

\newcommand{\muoz}{\mu_{1,0}}
\newcommand{\muzz}{\mu_{0,0}}

\newcommand{\muze}{\mu_{0,\e}}

\newcommand{\dtit}{\partial_{\tit}}

\newcommand{\trhoz}{\ti{\rho}_0}

\newcommand{\st}{\text{ s.t. }}

\newcommand{\tia}{{\ti{a}}}

\newcommand{\geps}{g_\e}

\newcommand{\trhoi}{\ti{\rho}_I}

\newcounter{thmm}

\newcommand{\cE}{{\mathcal E}}

\newcommand{\hz}{{\hat{\mathbf z}_\e}}

\newcommand{\qt}{{Q_T}}

\newcommand{\tse}{{\frac{t}{\e}}}
\newcommand{\Tse}{{\frac{T}{\e}}}

\newcommand{\om}{{\Omega}}
 \newcommand{\tet}{{D_t^\tau}}
\newcommand{\cyt}{{{\mathfrak Q}_T}}
\newcommand{\ocyt}{{\ov{\cyt}}}

\newcommand{\cK}{{\cal K}}
\newcommand{\cke}{\cK_\e}
\newcommand{\cked}{\cke^\delta}

\newcommand{\trhoze}{\ti{\rho}_{0,\e}}

\NewEnviron{thmm}{%
\refstepcounter{thmm}
  \begin{center}
\begin{tikzpicture}
  \node[rectangle, inner sep=5pt,minimum width=\textwidth,
    text=black!93!brown!93!yellow, text opacity=1,
    draw=green, ultra thick, draw opacity=1,
    fill=white, fill opacity=.7] 
  (box){%
    \begin{minipage}{.95\textwidth}
      \textbf{Theorem} \arabic{thmm}. \BODY
  \end{minipage}};
\end{tikzpicture}
\end{center}
}
 \newcounter{examples}
    
\newcounter{lemm}       
\NewEnviron{lemm}{%
\refstepcounter{lemm}
  \begin{center}
\begin{tikzpicture}
  \node[rectangle, rounded corners, inner sep=5pt,minimum width=\textwidth,
    text=black!93!brown!93!yellow, text opacity=1,
    draw=red, ultra thick, draw opacity=0.5,
    fill=white, fill opacity=.7] 
  (box){%
    \begin{minipage}{.95\textwidth}
      \textbf{Lemma} \arabic{lemm}. \BODY
  \end{minipage}};
\end{tikzpicture}
\end{center}
}
\newcounter{propm}       
\NewEnviron{propm}{%
\refstepcounter{propm}
  \begin{center}
\begin{tikzpicture}
  \node[rectangle, rounded corners, inner sep=5pt,minimum width=\textwidth,
    text=black!93!brown!93!yellow, text opacity=1,
    draw=blue, ultra thick, draw opacity=0.75,
    fill=white, fill opacity=.7] 
  (box){%
    \begin{minipage}{.95\textwidth}
      \textbf{Proposition} \arabic{propm}. \BODY
  \end{minipage}};
\end{tikzpicture}
\end{center}
}

\newcounter{defim}       
\NewEnviron{defim}{%
\refstepcounter{defim}
  \begin{center}
\begin{tikzpicture}
  \node[rectangle, rounded corners, inner sep=5pt,minimum width=\textwidth,
    text=black!93!brown!93!yellow, text opacity=1,
    draw=yellow, ultra thick, draw opacity=0.75,
    fill=white, fill opacity=.7] 
  (box){%
    \begin{minipage}{.95\textwidth}
      \textbf{Definition} \arabic{defim}. \BODY
  \end{minipage}};
\end{tikzpicture}
\end{center}
}

\newcounter{assum}       
\NewEnviron{assum}{%
\refstepcounter{assum}
  \begin{center}
\begin{tikzpicture}
  \node[rectangle, rounded corners, inner sep=5pt,minimum width=\textwidth,
    text=black!93!brown!93!yellow, text opacity=1,
    draw=gray!5,  thick, draw opacity=0.75,
    fill=gray!5, fill opacity=.7] 
  (box){%
    \begin{minipage}{.95\textwidth}
      \textbf{Assumptions} \arabic{assum}. \BODY
  \end{minipage}};
\end{tikzpicture}
\end{center}
}

\newcounter{rmkm}       
\NewEnviron{rmkm}{%
\refstepcounter{rmkm}
  \begin{center}
\begin{tikzpicture}
  \node[rectangle, rounded corners, inner sep=5pt,minimum width=\textwidth,
    text=black!93!brown!93!yellow, text opacity=1,
    draw=gray,   thick, dashed, draw opacity=0.275,
    fill=white, fill opacity=.7] 
  (box){%
    \begin{minipage}{.95\textwidth}
      \textbf{Remark} \arabic{rmkm}. \BODY
  \end{minipage}};
\end{tikzpicture}
\end{center}
}

\newcounter{corom}       
\NewEnviron{corom}{%
\refstepcounter{corom}
  \begin{center}
\begin{tikzpicture}
  \node[rectangle, rounded corners, inner sep=5pt,minimum width=\textwidth,
    text=black!93!brown!93!yellow, text opacity=1,
    draw=gray,  ultra thick, densely dashed, draw opacity=0.75,
    fill=white, fill opacity=.7] 
  (box){%
    \begin{minipage}{.95\textwidth}
      \textbf{Corollary} \arabic{corom}. \BODY
  \end{minipage}};
\end{tikzpicture}
\end{center}
}

\newcommand{\oqt}{{\ov{Q}_T}}
\newcommand{\expful}[3]{\exp\left( - \int_{#3}^0 \zteps(\bfx,#1+s,#2+\e s) ds \right)}
\DeclareMathOperator*{\esssup}{ess\,sup}
\newcommand{\esupx}{{\esssup\limits_{\bfx \in \Omega}\,}}
\newcommand{\esupxt}{{\esssup_{\bfx \in \Omega}\,}}
\newcommand{\expm}[1]{\exp(-\ztmin #1)}

\newcommand{\wcvg}{\rightharpoonup}
\newcommand{\wscvg}{\stackbin{\star}{\rightharpoonup}}
\newcommand{\chim}{\theta_m}

\newcommand{\ztp}{Z_T'}
\newcommand{\zt}{Z_T}
\newcommand{\ztpair}[1]{{Z_{#1}',Z_{#1}}}
\newcommand{\xtpair}[1]{{X_{#1}',X_{#1}}}
\newcommand{\cqt}{{{\mathfrak{Q}}_T}}

\newcommand{\zz}{{\mathbf z}_0}
\newcommand{\cL}{{\boldsymbol {\mathcal L}}}
\newcommand{\bfw}{{\boldsymbol w}}
\newcommand{\cA}{{\mathcal A}}
\newcommand{\bH}{{\mathbf H}}
\newcommand{\bL}{{\mathbf L}}
\newcommand{\bfz}{{\mathbf z}}
\newcommand{\bfZ}{{\mathbf Z}}
\newcommand{\bfv}{{\boldsymbol v}}

\newcommand{\bho}{{\mathbf H}^1(\om)}

\newcommand{\blo}{{\mathbf L}^2(\om)}

\newcommand{\zepsD}{\bfz_{\e,\Delta}}

\newcommand{\rhoeD}{\rho_{\e,\Delta}}
\newcommand{\orhoeD}{\ov{\rho}_{\e,\Delta}}

\newcommand{\sed}[1]{{\mu}_{\e}^{#1}}
\newcommand{\osed}[1]{\ov{\mu}_{\e}^{#1}}

\newcommand{\vepsd}[2]{{\mathbf  U}^{#1}_{\e,#2}}
\newcommand{\tepsd}[2]{{\mathbf T}^{#1}_{\e,#2}}

\newcommand{\cle}{\cL_\e}

\newcommand{\clen}{\cL_{\e}^n}
\newcommand{\cled}[1]{\cL_{\e}^{#1}}
\newcommand{\cleD}{\cL_{\e,\Delta}}

\newcommand{\muzed}[1]{\sed{#1}}
\newcommand{\omuzed}[1]{\osed{#1}}
\newcommand{\mueD}{\mu_{0,\e,\Delta}}

\newcommand{\lamd}[1]{\lambda^{#1}}

\newcommand{\pmr}{{\mathcal A}}
\newcommand{\lame}{\lambda_\e}
\newcommand{\lamed}[1]{{\lambda^{#1}_\e}}
\newcommand{\lamz}{\lambda_0}
\newcommand{\lameD}{\lambda_{\e,\Delta}}

\newcommand{\rhoed}[2]{\varrho^{#1}_{\e,#2}}
\newcommand{\orhoed}[2]{\ov{\varrho}^{#1}_{\e,#2}}

\newcommand{\ztepsd}[2]{\zeta^{#1}_{\e,#2}}

\newcommand{\zepsd}[1]{\bfZ^{#1}_{\e}}

\newcommand{\bepsd}[1]{\beta^{#1}_\e}

\newcommand{\zpd}[1]{{\bfZ}_p^{#1}}
\newcommand{\tizepsD}{\ti{\bfz}_{\e,\Delta}}

\newcommand{\xat}{\bfx,a,t}
\newcommand{\dxat}{d\bfx\,da\,dt}

\newcommand{\did}[2]{{\rm D}^{#1}_{#2}}
\newcommand{\tis}{\ti{s}}

\newcommand{\cnj}{{C^n_j}}
\newcommand{\ocnj}{\ov{C}^n_j}
\newcommand{\ucnj}{\underline{C}^n_j}

\newcommand{\varphin}{{\boldsymbol \varphi}^n}
\newcommand{\varphiD}{{\boldsymbol \varphi}_{\Delta}}

\newcommand{\cIe}{{\cal I}_\e}
\newcommand{\bvarphi}{{\boldsymbol \varphi}}
\newcommand{\caIe}{{{\cal I}_\e}}

\def\<{\langle}
\def\>{\rangle}

\usepackage{etoolbox}
\usepackage{amsmath,amsfonts,enumerate,amsbsy,amssymb}

\usepackage{mathtools, stmaryrd}
\usepackage{xparse}
\newcommand{\dom}{\{0,1\}}
\newcommand{\bfx}{{x}}

\DeclareMathOperator*{\argmin}{arg\,min}

\title{From delayed and constrained minimizing movements \\to the harmonic map heat equation}

\author{Vuk Mili\v si\' c \thanks{Laboratoire Analyse, G{\'e}om{\'e}trie \& Applications (LAGA),
Universit{\'e} Paris 13,
FRANCE ({\tt milisic@math.univ-paris13.fr}), Draft version of \today.}
}

\begin{document}

\maketitle


\pagestyle{myheadings}
\thispagestyle{plain}
\markboth{V. Mili\v si\'c }{From delayed minimizing movements to the harmonic map heat equation}

\begin{abstract} 
In the context of cell motility modelling and more particularly related to the Filament Based Lamelipodium Model \cite{OeSchVi,MR3385931,MR3675606}, 
this work deals with a rigorous mathematical proof of convergence between solutions of two problems :
we start from a microscopic description of adhesions using a delayed and constrained
vector valued equation with spacial diffusion
and show the convergence towards  
the corresponding
friction limit. 
The convergence is performed with respect to the bond characteristic lifetime $\e$ whose inverse is also proportional to the stifness 
of the bonds. The originality of this work is the extension of  gradient flow techniques to our setting. 
Namely, the discrete  finite difference term in the gradient flow energy is here replaced by a delay term
which complicates greatly the mathematical analysis.
Contrarily to the standard approach \cite{AmGiSa,Oel.11},
compactness in time is not provided by the energy minimization process~: a series of past times are taken into account 
in our discrete energy. A supplementary equation on the time derivative is obtained requiring uniform estimate with respect to  $\e$
of the Lagrange multiplier and 
provides compactness.
Due to the non-linearity induced by the constraint, a specific stability estimate useful in our previous works, 
is not at hand here.  Numerical simulations even showed that this estimate does not hold. 
Nevertheless, transposing our delay operator, we succeed in  proving convergence under slightly weaker hypotheses.
The result relies on a careful initial layer analysis, extending \cite{mi.proc} to the space dependent setting.
\end{abstract}


\ccode{
35Q92 
}


\section{Introduction}

Cell motility is at heart of important biological/medical concerns  (cancer metastasis, wound healing, etc.) \cite{bray2000cell}.
Among models describing spontaneous motion of cells, two types appear : those who heuristically 
mimic macroscopic features and models based on a microscopic description 
that are in some sense homogenized. The Filament Based Lamelipodium Model (FBLM) \cite{OeSchVi} belongs
to the second category and has reached a certain level of maturity \cite{MR3385931,MR3675606}.

Adhesion mechanisms are some of the pillars of the FBLM and appear as friction terms. 
In the pioneering paper \cite{OeSchVi}, they are obtained as formal limits of memory terms
inside the Euler-Lagrange equations associated to a minimization process. 
This limit is interpreted as quasi-instantaneous with respect to a dimensionless 
parameter $\e$.
Our work
deals with the rigorous mathematical justification of this asymptotic.

Previously, we introduced  simplifications that allowed 
to fully understand, from the mathematical point of view, either
the delay model, for fixed $\e$, or its convergence when $\e$ tends
to zero \cite{MiOel.1,MiOel.2,MiOel.3,MiOel.4}. More specifically 
in \cite{MiOel.1} and \cite{MiOel.2}, we studied  adhesions of a
single point submitted to an external force and proved convergence.
In \cite{MiOel.3} we proved that a non-linear fully coupled model
could either have global solutions or, if the external load exceeds 
 the microscopic adhesion capacity, blow-up could occur. 
More recently \cite{MiOel.4},
we extended these results adding space dependent adhesion and diffusion. 

In the previous works, Euler-Lagrange equations were considered and
in some cases \cite{MiOel.4} this is equivalent to the minimization 
of a convex energy. Here we consider the minimization process 
for which the energy functional contains adhesion terms and the Dirichlet
energy. But it is set pointwisely on the sphere almost everywhere, leading to 
a non-linear saddle point problem at the Euler-Lagrange level. 
The mathematical tools previously introduced extend only very partially 
 to this new problem.

Gradient flow techniques  provide existence of solutions
for complicated possibly non-linear energies complemented 
with a finite difference term in time \cite{AmGiSa,GiMaTo}. Here the delay term
in the energy could be considered as a generalization of such a finite difference.
Except that it provides neither existence of solutions nor compactness in time.
Thus we are forced to discretize the energy with respect to time and age. 
Here the age accounts for the delay.
First, we obtain  new energy estimates similar
to the minimization principle in gradient flow theory (cf. Lemma \ref{lem.nrj}), 
but there is then an extra amount of work in order to prove compactness {\em i.e.}
boundedness of the time derivative in an appropriate space (cf. Proposition \ref{prop.time.compactness}).
This estimate is made possible thanks to a closed equation
obtained for the discrete time derivative of the position. This equation appears 
when taking finite differences  with respect to time of the Euler-Lagrange equations 
of the minimization process. 
 In \cite{MiOel.4}, another estimate of the elongation provided extra
 compactness useful in the asymptotic of the variational 
 formulation that is not at hand
 here. 
 The  reason will be made more precise below.
 Thus, we were forced to transpose the delay 
 term in the Euler-Lagrange equations on the test function
 and the density of population of bonds $\rhoe$. This latter unknown 
 is singular~: $\dt \rhoe$ is a measure that
 converges, when $\e$ goes to zero, to $\dt \rhoz$, the time derivative
 of the population of linkages and to a  Dirac mass located 
 near the origin in time for all ages. Since the problem is here space dependent, 
 these results extend and complete the initial layer analysis performed  in \cite{mi.proc}.

To be more specific, 
we denote by $\Omega:=(0,1)$. 
The vector  position in $\R^d$ of the moving binding site, $\zeps(\bfx,t)$,  minimizes at each time $t\geq 0$ an energy functional~:
\begin{equation}\label{eq.minimiz}
 \zeps(\bfx,t) = \argmin_{w \in \pmr} {\cal E}_t(w) ,
\end{equation}
where the minimization is performed on the set
$$
\cA:=\left\{ \bfw \in \bho \st |\bfw(\bfx)|^2 =1,\; \text{a.e.}\; \bfx \in \om\right\}.
$$
The energy is defined for every $\bfw\in \cA$ as 
\begin{equation}\label{eq.nrj}
\begin{aligned}
  & {\cal E}_t( \bfw(\cdot)) := 
 \frac{1}{2\e} \int_\om \int_{\rr} \frac{| \bfw(\bfx)- \zeps(\bfx,t-\e a) |^2 }{\e} \rhoe(\bfx,t,a) da d\bfx +
\frac{1}{2} \int_\om | \dx \bfw |^2 d\bfx.
\end{aligned}
\end{equation}
Past positions are given by the function $\zeps(\bfx,t)=\zp(\bfx,t)$ for $t<0$. 
The
age distribution $\rhoe=\rhoe(\bfx,a,t)$ is the solution of the structured model~:
\begin{equation}
\label{eq.rho.eps} 
\left\{ 
\begin{aligned} 
&\varepsilon \partial_t \rhoe + \partial_a \rhoe+ \zteps \, \rhoe = 0 
\,, &\bfx \in \Omega, \, a>0 \, , \;t > 0 , \\ 
& \rhoe(\bfx,a=0,t)=\beps(\bfx,t)\left(1-\muze(t,\bfx)  \right) 
\, , &\bfx \in \Omega,\; a=0,\, t > 0  , \\
& \rhoe(\bfx,a,t=0)=\rhoi(\bfx,a)
\, ,& \bfx \in \Omega, a>0, t=0,
\end{aligned}  
\right. 
\end{equation}
where $\muze(\bfx,t):=\int_0^\infty \rhoe(\bfx,\tia,t) \, d \tia$ and the on-rate of bonds is a given
function $\beps$ times a factor, that takes into account saturation of the moving binding site with linkages.
When the off-rate $\zteps$ is a prescribed function, we say that the problem is weakly coupled~:
first one exhibits $\rhoe$ solving  \eqref{eq.rho.eps} which then becomes  the  weight in \eqref{eq.nrj}.

First, we discretize in time and age the minimization process \eqref{eq.minimiz} and 
the age structured system \eqref{eq.rho.eps}. For the transport problem \eqref{eq.rho.eps}
we use $i)$ the upwind scheme inside the domain, $ii)$ an implicit discretization of the off-rates,
$iii)$ the non-local term is discretized using a piecewise constant approximation.
This step provides, as in the gradient flow case (see for instance minimizing movements chap. 2 \cite{AmGiSa}), 
existence of a discrete pair of solutions $((\rhoed{n}{j})_{j\in\N},\zepsd{n})_{n\in\N}$.
Then thanks to compactness arguments, we pass to the limit with respect to  the discretization 
parameter $\Da$, and prove that there exists a unique couple $(\rhoe,\zeps)$. 
The bond population density $\rhoe$ solves \eqref{eq.rho.eps}, whereas
$\zeps$ satisfies, almost everywhere in $(0,T)$,  
the weak formulation associated to the Euler-Lagrange equation :
\begin{equation}\label{eq.delayed.map}
\left\{
 \begin{aligned}
& \cL_\e -  \dxx \zeps + \lame  \zeps  = 0 , & \text{ a.e. }(\bfx,t) \in \om\times(0,T),\\
& | \zeps(\bfx,t) | =1 &  \text{ a.e.} (\bfx,t) \in \om\times(0,T),\\
 & \ddn{}\zeps =0 & \forall  (\bfx,t) \in \dom \times (0,T), \\
 & \zeps(x,t) = \zp(\bfx,t)& \text{ a.e. }(\bfx,t) \in \om \times \RR_-,
\end{aligned}
\right.
\end{equation}
where $\cL_\e(\bfx,t):=\int_\rr \rhoe (\bfx,a,t) (\zeps(\bfx,t)-\zeps(\bfx,t-\e a))/\e da$ and $\lame$ is the Lagrange multiplier associated
to the constraint $ | \zeps(\bfx,t) | =1$. 
We prove that when $\e$ goes to zero, $(\rhoe,\zeps)$ 
the solutions of the previous minimization problem converge to $(\rhoz,\zz)$. 
These solve the limit problems  reading~: 
\begin{compactitem}
\item  Find $\zz \in C^0_t([0,T]; \bH^1_\bfx (\om)) \cap H^1_t ((0,T);\bL^2_x(\om))$ being the weak solution of the problem
\begin{equation}\label{eq.heat.harmonic.map}
\left\{
 \begin{aligned}
&  \muoz \dt \zz -  \dxx \zz - | \dx \zz |^2  \zz  = 0 , & \text{ a.e. }(\bfx,t) \in \om\times(0,T),\\
& | \zz(\bfx,t) | =1 &  \text{ a.e.} (\bfx,t) \in \om\times(0,T),\\
 & \ddn{}\zz =0 & \forall  (\bfx,t) \in \dom \times (0,T), \\
 & \zz(x,0) = \zp(\bfx,0)& \text{ a.e. }(\bfx,t) \in \om \times \{0\}.
\end{aligned}
\right.
\end{equation}
The convergence of $\zeps$ towards $\zz$ holds strongly in $C(\om\times[0,T])$.
\item The function $\mu_{k,0} := \int_{\rr} a^k \rhoz (\bfx,a,t) \, da$ represents the moment of order $k$ of $\lim_{\e\to 0} \rhoe =: \rhoz$ which solves in $C([0,T];L^1_a(\rr;L^\infty_\bfx(\om)))$ 
\begin{equation}\label{eq.rho.zero}
\left\{
 \begin{aligned}
 &  \partial_a \rhoz+ \ztz \, \rhoz = 0 
\,, &\bfx \in \Omega , \;a>0, \; t>0, \\ 
& \rhoz(\bfx,a=0,t)=\bz(\bfx,t)\left(1-\muzz(\bfx,t)  \right) 
\, , & \bfx \in \Omega , \;a=0, \; t>0.\\
\end{aligned}
\right.
\end{equation}
The solution $(1+a)\rhoe$ converges to $(1+a)\rhoz$ in $L^1(\rr\times(0,T);L^\infty_\bfx(\om))$ strong.
\end{compactitem}

The article is structured as follows. In Section \ref{sec.hypo}, we  list the hypotheses
used throughout the paper and set notations. 
In Section \ref{sec.a.priori}, we detail the discrete minimization
process in age and time providing the piecewise solutions $((\rhoed{n}{j})_{j\in\N},\zepsd{n})_{n\in\N}$. 
In the same section, we provide stability estimates  in the appropriate functional
spaces. We underline that most of the results obtained therein are uniform with respect to $\e$ and 
the discretisation step $\Delta a$~: the same properties can be extended to the continuous model for fixed $\e$.
This leads to study first this latter limit for $\e$ fixed and $\Delta a$ going to 0. This is done in Section \ref{sec.conv.eps.fixed}.
Then when $\e$ tends to zero, we prove, in Section \ref{sec.conv.e}, that indeed convergence occurs
towards the limit heat harmonic map equation \eqref{eq.heat.harmonic.map} . As $\dt \zeps$ converges weakly in  $L^2(\om\times(0,T))$
and $\rhoe$ converges strongly in $L^1(\rr\times(0,T);L^\infty(\om))$,  it is not possible to obtain directly the
convergence of the delay term $\cL_\e(\bfx,t)$ towards $\muoz \dt \zz$. 
Instead, as mentioned above, we transpose the delay operator on a test function 
and on $\rhoe$,  and then we pass to the limit with respect to  $\e$. 

\section{Notations and hypotheses}\label{sec.hypo}

We set $\qt := \RR^*_+ \times (0,T)$ and $\oqt:=\rr\times[0,T]$.  
As stated in the introduction we consider a one dimensional space domain $\Omega:=(0,1)$.
We set as well $\cqt := \om\times\rr\times(0,T)$.

We define  $X_T := C^0(\oqt;L^1(\Omega))$ to be 
the Banach space of continuous functions in age and time whose $L^1$ norm 
in space goes
to zero when $a$ goes to infinity. We endow $X_T$ with the norm :
$$
\nrm{f}{X_T} := \sup_{(a,t)\in\oqt} \nrm{f(\cdot,a,t)}{L^1(\Omega)}.
$$
$X_T$ is a Banach space \cite{simon}. 
It is also a closed subspace of $Z_T:=C_b(\oqt;L^1(\om))$ which is a non-separable Banach space.
We define $Y_T:= L^1(\qt;L^\infty(\Omega))$ which is also a Banach space endowed with the
 corresponding norm :
$$
\nrm{f}{Y_T} := \int_{\qt}  \nrm{ f (\cdot,a,t) }{L^\infty(\Omega)} da dt.
$$

We denote  the discrete differences as 
$$
\did{\tau}{t} f :=  \frac{f(\bfx,a,t+\tau)-f(\xat)}{\tau},\quad \did{\alpha}{a} f:=  \frac{f(\bfx,a+\alpha,t)-f(\xat)}{\alpha},
$$
and we define the space of Banach valued functions 
$$
U_T:= \left\{ f \in Y_T\st 
\limsup_{\sigma \to 0} \left( \nrm{\did{\sigma}{a}f}{Y_{(T-\sigma)}} + \nrm{\did{\sigma}{t} f}{Y_{(T-\sigma)}} \right) < \infty
\right\}
$$
and one endows $U_T$ with the norm~:
$$
\nrm{f}{U_T} := \nrm{f}{Y_T} +  \limsup_{\sigma \to 0} \left( \nrm{\did{\sigma}{a}f}{Y_{(T-\sigma)}} + \nrm{\did{\sigma}{t} f}{Y_{(T-\sigma)}} \right).
$$
If the same space is set on a time interval $(t_1,t_2)$ then the notation $U_{(t_1,t_2)}$ is well understood.
In the rest of the paper we abbreviate the notation of function spaces writing the subscripts $t$
for function spaces on $t\in[0,T]$  and the subscript $\bfx$ for function spaces on $\bfx \in \om$, 
for instance $C_t L^1_a L^\infty_\bfx$ denotes $C([0,T];L^1(\rr;L^\infty(\om)))$. 

One should notice that $U_T$ is in fact the space of 
functions of bounded variation with values in a Banach space $BV(\rr\times(0,T);L^\infty(\om))$.
The generic  space $BV((0,T);Z^*)$ is 
presented and studied in a very detailed way in \cite{HeiPatRen.19}, 
where $Z^*$ is the dual space of a Banach space $Z$.
Since $L^\infty(\om)$ is  the dual space of $L^1(\om)$ we are exactly in this framework.
The semi norms
with the discrete derivatives coincide with the total variation of the $L^\infty$-valued
Radon measures corresponding to the time/age derivative {\em i.e. }
$$
|\dt f| := \sup_{
\substack{\phi \in \D(Q_T;L^1(\om)) \\
\nrm{\phi}{L^\infty_{a,t}L^1_\bfx}\leq 1}} 
\int_{\rr\times (0,T)} \langle f , \dt \phi\rangle da dt,\quad  |\da f| := \sup_{
\substack{\phi \in \D(Q_T;L^1(\om)) \\
\nrm{\phi}{L^\infty_{a,t}L^1_\bfx}\leq 1}} 
\int_{\rr\times (0,T)} \langle f , \da \phi\rangle da dt
$$
where the brackets $\langle \cdot,\cdot\rangle$ denote the $L^\infty(\om),L^1(\om)$ duality in the space variable. The proof of
this equivalence can be found in \cite{Dafermos.Book.2000} p. 12 in the proof of Theorem 1.7.1 and is easily extendable
to the Banach valued case presented in \cite{HeiPatRen.19}.


\begin{hypo}\label{hypo.data}
The dimensionless parameter $\varepsilon > 0$ is assumed to induce two families of chemical rate functions that satisfy:
\begin{enumerate}[(i)]
\item For every $\e\geq 0$, the  function $\beps$ belongs to $W^{1,\infty}
(Q_T)$ and the off-rate $\zteps$ is s.t. 
$\zteps \in W^{1,\infty}(\Omega\times\rr\times[0,T])$ moreover itr holds that 
\begin{equation*}
\nrm{ \zteps-\ztz}{L^{\infty}_{\bfx,a,t}} \to 0 \quad \text{and} \quad \nrm{ \beps - \bz}{L^\infty_{\bfx,t}}  \to 0
\; 
\end{equation*}
as ${\varepsilon} \to 0$.
\item We also assume that there are upper and lower bounds such that
\begin{equation*}
0 < \zeta_{\rm min} \leq \zteps(\bfx,a,t) \leq \zeta_{\rm max} \quad \text{and} \quad 
0 < \beta_{\rm min} \leq \beps(\bfx,t) \leq \beta_{\rm max}
\end{equation*}
for all $\varepsilon>0$, $\bfx \in \Omega$, $a \geq 0$ and $t>0$. 
\end{enumerate}
\end{hypo}

The initial data for the density model \eqref{eq.rho.eps} satisfies some hypotheses that we sum up here.
\begin{hypo}\label{hypo.data.deux}
  The initial condition $\rhoi \in L^\infty_{\bfx,a}(\Omega\times\rr)$ satisfies
\begin{itemize}
\item positivity and boundedness~: there exists $M > \bmax$, s.t. 
$$
M\geq \rhoi(\bfx,a)\geq0 \; ,\quad \text{ a.e. }  (\bfx,a) \in \Omega\times\rr\; ,
$$
moreover, one has also that the total initial population satisfies
$$
0< \int_{\rr} \rhoi(\bfx,a) da < 1
$$
for almost every $\bfx \in \Omega$.
\item boundedness from below  of the zero order moment,
$$
0< \mu_I :=\int_{\rr} \rhoi (\bfx,a) \; da ,\quad \text{for a.e.} \; \bfx \in \om \; ,
$$
\item initial integrability with respect to the limit problem :
$$
\int_{\rr} \sup_{\bfx \in \Omega}  \rhoi(\bfx,a) a^p da < \infty,\quad \text{for } p\in\{0,1,2\}   \; ,
$$
\item the derivative with respect to  age satisfies as well :
$$
\limsup_{\sigma \to 0} \int_{\rr} \sup_{\bfx \in \Omega} \left| \did{\sigma}{a}  \rhoi(\bfx,a) \right| da < \infty.
$$
\end{itemize}

\end{hypo}
Concerning the minimization problem \eqref{eq.minimiz}, we assume
\begin{hypo}\label{hypo.data.trois}
 The past data satisfies~:
\begin{enumerate}[i)]
\item for every time $t\leq0$, we assume that  $\zp(\cdot,t)$ is in $\cA$, 
 \item  there exists a Lipschitz constant which is $L^2$ in space s.t.~:
 \begin{equation}\label{eq.lip.past}
  | \zp(\bfx,t_2) - \zp(\bfx,t_1) | \leq C_{\zp}(\bfx)|t_2-t_1|, \quad \forall  (t_2,t_1) \in (\RR_-)^2
%
\end{equation}
for a.e. $\bfx \in \om$ where $C_{\zp}(\bfx) \in L^2(\Omega)$.
\end{enumerate}
\end{hypo}



\section{Existence of minimizers and {\em a priori} estimates : the discrete scheme}\label{sec.a.priori}

%

We discretize both \eqref{eq.rho.eps} and the minimization 
process \eqref{eq.minimiz}  in time and age, but not in space.
We set $\Delta a$ a small parameter denoting the age discretization step, 
while the time step satisfies the CFL condition $\Delta t = \e \Delta a$. 
This provides 
  $N:=\lfloor T/\Delta t\rfloor$, the number of times steps.
We solve~:
\begin{itemize}
\item for the $\rhoe$ model, we use a first order upwind scheme and treat the source term implicitly, so  we define inside the mesh
\begin{equation}\label{eq.def.scheme.rho.int}
 \rhoed{n+1}{i}(\bfx) := \rhoed{n}{i-1}(\bfx)  /\left( 1 +\frac{\Delta t}{\e} \ztepsd{n+1}{i} (\bfx) \right), \quad i \in \N^*,\quad n \in \N\cup \{-1\},
\end{equation}
while on the boundary we set 
\begin{equation}\label{eq.def.scheme.rho.bdry}
 \rhoed{n+1}{0} := \rhoed{n+1}{b} /\left( 1 +\frac{\Delta t}{\e} \ztepsd{n+1}{0} \right), \quad n\in\N\cup \{-1\},
\end{equation}
where 
$$
\rhoed{n+1}{b} := \bepsd{n+1}(1-\sed{n+1}),\quad
\sed{n+1} := \sum_{i=0}^\infty \rhoed{n+1}{i} \Delta a.
$$
This definition provides explicitly  $\rhoed{n+1}{0}$,
$$
\rhoed{n+1}{0} := \frac{\bepsd{n+1}}{1 + \Delta a (\bepsd{n+1}+\ztepsd{n+1}{0})} \left(1 - \sum_{i=1}^\infty \rhoed{n+1}{i} \Delta a\right).
$$
The initial condition is defined as
$$
\rhoed{-1}{i}:= \frac{1}{\Delta a} \int_{i\Delta a}^{(i+1)\Delta a} \rhoi(a) da, \quad \forall i \in \N.
$$
The zero order moment $\sed{n+1}:=\Delta a\sum_{i\in\N} \rhoed{n+1}{i}$ can be expressed in an inductive way :
\begin{equation}\label{eq.s}
 \sed{n+1} + \Da \sum_{i=0}^\infty \frac{\Dt}{\e} \ztepsd{n+1}{i} \rhoed{n+1}{i} = \sed{n} + \rhoed{n+1}{b}= \sed{n} + \Da \bepsd{n+1}(1-\sed{n+1}).
\end{equation}
We define a piecewise constant function
$$
\rhoeD(\bfx,a,t) := \sum_{i,j\in\N^2} \rhoed{n}{j}(\bfx) \chiu{(i\Delta a,(i+1)\Delta a)\times(j\Delta t,(j+1)\Delta t)}(a,t).
$$
\item whereas the minimization process is performed for each $n \in \N$ 
\begin{equation}\label{eq.disc.min}
  \zepsd{n} :=   \argmin_{\bfw \in \cA}  \cE_{n}(\bfw), 
\end{equation}
 where the discrete energy functional reads :
  $$
 \begin{aligned}
 \cE_{n}(\bfw) :=  &  \nud \int_{\om} |  \dx \bfw |^2 d\bfx  + \frac{\Delta a}{4 \e} \left\{  \int_\om  (\bfw-\zepsd{n-1})^2 \rhoed{n}{0} d\bfx \right. \\
  & \quad \quad \left. +  \sum_{j=1}^{\infty} \int_\om  \left(  (\bfw - \zepsd{n-j})^2+ (\bfw - \zepsd{n-j-1})^2 \right)\rhoed{n}{j}  d \bfx \right\}  
  \end{aligned} 
$$
for all $n \in \N$
and $\zpd{i} := \int_{i\Delta t}^{(i+1) \Delta t} \zp(\bfx,t)dt/\Delta t$ for all $i\in \Z$, $i<0$, and we set $\zepsd{n}=\zpd{n}$ for every $n<0$. 
We define  the piecewise constant function
$$
\zepsD(\bfx,t) := 
\sum_{n=-\infty}^\infty  
\zepsd{n} (\bfx) 
\; \chiu{(n\Delta t, (n+1)\Delta t)}(t).
$$
The  piecewise linear extension reads~:
$$
\tizepsD := \sum_{n\in\N} \left\{ \zepsd{n} + \left( \frac{t}{\Delta t}-n\right) \delta \zepsd{n+\nud} \right\} \chiu{(n,(n+1))\Delta t}(t),
$$
where  $\delta \zepsd{n+\nud} := \zepsd{n+1}-\zepsd{n}$, while in what follows we denote as well
$\delta \rhoed{n}{i+\nud} := \rhoed{n}{i+1}-\rhoed{n}{i}$ an so on.

\end{itemize}
\subsection{Positivity and convergence  of the discrete solution $\rhoeD$} 

From Lemma \ref{lem.muze.bounds} to Theorem \ref{thm.cvg.dicrete.rhoe}, we extend
results from previous works \cite{MiOel.1,MiOel.4,mi.proc} to the  discrete case.
When needed, we characterize also some properties of $\rhoe$, the continuous solution of \eqref{eq.rho.eps}.


\begin{lemma}\label{lem.muze.bounds}
For almost every $x\in\om$, under the CFL condition $\Dt = \e \Da$, 
under hypotheses \ref{hypo.data}, if
 $\rhoed{-1}{i}\geq 0$ and $\sed{-1}\leq1$ then 
 $$
 \rhoed{n}{i}\geq 0,\quad 0\leq \sed{n}\leq 1,\quad \forall (n,i) \in \N^2.
 $$
Moreover if there exists a constant $0<\mumin < \min( \sed{-1}, \bmin/(\bmin+\ztmax))$, then
$\sed{n}>\mumin$ for every $n\in\N$,
\end{lemma}
\begin{proof}
 The first result is proved by induction : by hypothesis, the claim is true for $k=-1$.
We assume that for $k=n$, $\rhoed{n}{i} \geq 0$ and  $\sed{n} \in [0,1]$.
 Since $\rhoed{n}{i}\geq0$ for $i\in\N$, it is straightforward that $\rhoed{n+1}{i} \geq 0$
 for all $i \in\N^*$. Thanks to \eqref{eq.s}, one writes :
 $$
 (1-\sed{n+1})-\Da \sum_{i=0}^\infty \frac{\Dt}{\e} \ztepsd{n+1}{i}\rhoed{n+1}{i} = (1 -\sed{n})-\Da \bepsd{n+1}(1-\sed{n+1})
 $$
 which gives :
 $$
 \begin{aligned}
 (1+\Da \bepsd{n+1}) & (1-\sed{n+1}) -(1-\sed{n}) = \Da \sum_{i=0}^\infty \frac{\Dt}{\e} \ztepsd{n+1}{i}\rhoed{n+1}{i} \\
 & \geq \frac{\Dt}{\e} \Da  \ztepsd{n+1}{0}\bepsd{n+1}(1-\sed{n+1})/\left( 1 +\frac{\Delta t}{\e} \ztepsd{n+1}{0} \right),
 \end{aligned}
  $$
 which rearranging terms on both sides provides 
 $$
\left(1+\Da \bepsd{n+1}-\Da^2 \bepsd{n+1}\ztepsd{n+1}{0}/\left( 1 + \Da \ztepsd{n+1}{0} \right)\right) (1-\sed{n+1}) \geq 1-\sed{n} \geq 0.
$$
This shows that $1-\sed{n+1} \geq 0$ since the coefficient in front of $(1-\sed{n+1})$ is always positive definite.
In turn one concludes that $\rhoed{n+1}{0}\geq 0$.

 We prove the last claim by induction, under the hypothesis on $\mumin$,
 the claim is true for $k=-1$. We suppose that the claim is true for $k=n$. Using \eqref{eq.s}
 gives 
$$
\begin{aligned}
& (\sed{n+1}-\mumin) (1+\Delta a (\ztmax + \bepsd{n+1})) \\
& \quad \geq  (\sed{n}-\mumin) + \Delta a \bepsd{n+1}(1-\mumin)-\ztmax\Delta a \mumin  > (\sed{n}-\mumin),
\end{aligned}
$$
where the latter inequality holds since $\mumin \leq  \bmin/(\bmin+ \ztmax)< 1$. Because the right hand side is strictly positive, so is the left hand side.
This shows the statement for $k=n+1$, and the recursion is complete.
 \end{proof}

Using the same Lyapunov functional $\cH[u]:=\int_\rr | u (a)| da + \left| \int_\rr u(a) da \right|$, as in \cite{MiOel.1,MiOel.4}, one proves that
\begin{proposition}\label{prop.exist.cont.rho.eps}
Under hypotheses \ref{hypo.data} and \ref{hypo.data.deux}, 
 there exists  a unique solution $\rhoe \in Y_T$, solving \eqref{eq.rho.eps}. Moreover~:
 $$
 \nrm{ \did{\tau}{t} \rhoe }{Y_T} < C \left( \nrm{\did{\tau}{a} \rhoi}{L^1_a L^\infty_\bfx} + \nrm{\beps}{W^{1,\infty}_{\bfx,t}}+ \nrm{\zteps}{W^{1,\infty}_{\bfx,a,t}} + \nrm{\rhoi(\cdot,0)-\beps(\cdot,0)}{L^\infty(\om)} \right),
 $$
 where the constant is independent of $\e$ and on $\tau$.
\end{proposition}

\begin{proof}
For the existence and uniqueness part, one proceeds as in Theorem 3.1 in \cite{MiOel.4}~: 
 as $\bfx$ is a mute parameter, for a.e. $\bfx \in \om$,  
there exists a solution $\rhoe(\bfx,\cdot,\cdot) \in C_t([0,T];L^1_a(\rr))$. Then using 
 Duhamel's formula in order to commute the supremum with respect to  $\bfx$ with the integrals, one 
obtains the result in $Y_T$.
  Combining results from the proof of Lemma 5.1. p. 16 \cite{MiOel.4} and from Theorem 3.2 \cite{mi.proc}, one gets :
 $$
 \nrm{ \tet \rhoe }{L^\infty_\bfx(\Omega;L^1_{a,t}(\qt))} < C.
 $$
Indeed,  again, since $\bfx$ is only a mute parameter, one obtains easily that 
 $$
 \cH[{\tet \rhoe}(\bfx,\cdot,t)] \leq C_1 \exp(- \ztmin t /\e ) /\e + C_2,
 $$
 which then integrated in time and taking the ess-sup on $\Omega$ proves this first step.
 Then we use the method of characteristics and write~:
 $$
 \tet \rhoe (\bfx,a,t) := 
\begin{cases}
\begin{aligned}
 \tet & \rhoe(\bfx,0,t-\e a) \expful{a}{t}{-a} +\\
& + \int_{-a}^0 \expful{a}{t}{\tau} \cR_\tau (\bfx,a+\tau,t+ \e \tau) d\tau
\end{aligned} & \text{ if } t\geq\e a, \\
\begin{aligned}
 \tet & \rhoe(\bfx,a-t/\e ,0) \expful{a}{t}{-t/\e} +\\
& + \int_{-t/\e}^0 \expful{a}{t}{\tau} \cR_\tau (\bfx,a+\tau,t+ \e \tau) d\tau
\end{aligned} & \text{ if } t\leq\e a, \\
\end{cases}
$$
where $\cR_\tau(\bfx,a,t) := \tet \zteps(\bfx,a,t) \rhoe(\bfx,a,t)$. Now we define $q(a,t):=\esupx \left| \tet \rhoe(\bfx,a,t) \right|$.
One has 
$$
\begin{aligned}
 \int_0^\tse q(a,t) da \leq & \int_0^\tse q(0,t-\e a) da + C \int_0^\tse \esupx \int_0^a \exp(-\ztmin \tau) \rhoe(\bfx,a-\tau,t-\e \tau) d\tau 
da \\
& =: I_1 +I_2
\end{aligned}
$$
then
$$
\begin{aligned}
I_2 & \leq \ue \int_0^\tse \int_{t-\e a}^t \exp(-\ztmin (t-\tit)/\e) \esupx \rhoe(\bfx,a-(t-\tit)/\e,\tit) d \tit  da \\
& = \ue \int_0^t \int_{\frac{t-\tit}{\e}}^\tse \exp(-\ztmin (t-\tit)/\e) \esupx \rhoe(\bfx,a-(t-\tit)/\e,\tit)   da d \tit \\
& = \ue \int_0^t \exp(-\ztmin (t-\tit)/\e)  \int_0^{\frac{\tit}{\e}} \esupx \rhoe(\bfx,a,\tit) da d\tit < C.
\end{aligned}
$$
For the term $I_1$, one has $\tet \rhoe (x,0,t) = (\tet \beps) (1-\muze) - \beps \tet \muze$ which gives 
$$
| q(0,s) | \leq \bmax \esupx \int_{\rr} \left| \tet \rhoe(\bfx,a,t) \right| da  + \left| \tet \beps \right| \leq \bmax \exp(-\ztmin t /\e) / \e + C,
$$
so that 
$$
I_1 = \ue \int_0^t q(0,s) \expm{(t-s)/\e} q(0,s) ds \leq C + t/\e \expm{t/\e}/\e.
$$
These two estimates guarantee that 
$
\int_0^T \int_0^\tse q(a,t) da dt < C.
$
In a similar way one writes that
$$
\begin{aligned}
 \int_\tse^\infty q(a,t) da&  \leq \int_\tse^\infty \esupx \left| \frac{\rhoe(\bfx,a-t/\e,\tau)-\rhoe(\bfx,a-t/\e,0)}{\tau}\right| da \expm{t/\e} \\
 & + \int_\tse^\infty\int_{-t/\e}^0 \exp(\ztmin \tau)  \esupx  \rhoe(\bfx,a+\tau,t+\e\tau) d\tau \\
 &  \leq \int_0^\infty \esupx \left| \frac{\rhoe(\bfx,a,\tau)-\rhoe(\bfx,a,0)}{\tau}\right| da \expm{t/\e} \\
 & + \int_0^\tse \expm{\tau} \int_\rr \esupx \rhoe(\bfx,a,t-\e a\tau ) da d\tau, 
\end{aligned}
$$
the latter term being under control, we focus on the first one, that we denote $I_3$.
$$
I_3 = \expm{t/\e} \left\{ \left(\int_0^{\tau/\e} + \int_{\tau/\e}^\infty \right)\left| \frac{\rhoe(\bfx,a,\tau)-\rhoe(\bfx,a,0)}{\tau}\right| da \right\} =: I_{3,1}+I_{3,2}.
$$
Since $\rhoe$ is bounded uniformly in space and with respect to  $\e$, the first term $I_{3,1}$ is smaller than $\expm{t/\e}M/\e$.
Then using the method of characteristics, one splits $I_{3,2}$ in two parts :
$$
\begin{aligned}
 I_{3,2} \leq & \expm{t/\e}\left\{ \int_{\tau/\e}^\infty \esupx | \rhoi(\bfx,a-\tau/\e)-\rhoi(\bfx,a) |/\tau da \right. \\
 &  \left. + \frac{1}{\tau} \int_{\tau/\e} \esupx \rhoi(\bfx,a) \left| \expful{a}{\tau}{\tau/\e} -1 \right|  da \right\} \\
 & \leq \expm{t/\e} /\e  \left( \limsup_{h \in \rr}\int_\rr \esupx \left|  \frac{\rhoi(\bfx,a+h)-\rhoi(\bfx,a)}{h} \right| \right. \\
 & \left. \hspace{2cm}+   \nrm{\zteps}{W^{1,\infty}(\qt;L^\infty(\om))}\int_\rr \esupx \rhoi(\bfx,a)da da \right).
\end{aligned}
$$
This shows that $\int_0^T \int_\tse^\infty q(a,t)dadt < C$ which ends the proof.
\end{proof}

Then using standard {\em a priori} estimates provides in a similar manner 
as in the previous proof :
\begin{proposition}
 Under the previous hypotheses, one has as well that
 $$
\limsup_{\sigma\to 0} \nrm{\did{\sigma}{a} \rhoe}{Y_T} \leq C,
 $$
 where the constant is uniform with respect to  $\e$. This result together with the previous proposition 
 shows that $\rhoe \in U_T$ uniformly with respect to  $\e$.
\end{proposition}

One defines $C^n_j := (j\Delta a,(j+1)\Delta a)\times(n\Delta t,(n+1)\Delta t)$, and one sets
$$
\orhoed{n}{j}(\bfx) := \frac{1}{|C^n_j|} \int_{C^n_j} \rhoe(\bfx,a,t) da dt,
$$
where $\rhoe$ is the exact solution of \eqref{eq.rho.eps} and
$$
\orhoeD(\xat) := \sum_{(j,n)\in\N^2} \orhoed{n}{j}(\bfx) \chiu{(j\Delta a,(j+1)\Delta a)\times(n\Delta t,(n+1)\Delta t)}(a,t).
$$
in the same way one defines $\omuzed{n} = \int_{n\Dt}^{(n+1)\Dt} \muze(\bfx,t) dt /\Dt$, and so on.
With these notations, we compute error estimates for the upwind scheme :

\begin{lemma}\label{lem.cvg.rhoeD}
Under the same hypotheses as above, if $\rhoe\in U_T$ solves \eqref{eq.rho.eps}, and $\rhoeD$ is its piecewise constant approximation 
  computed using the upwind scheme \eqref{eq.def.scheme.rho.int} with the non-local boundary term \eqref{eq.def.scheme.rho.bdry}, 
   then one has
$$
\nrm{\rhoeD - \orhoeD}{Y_T} \leq O(\Delta a), \quad \nrm{\orhoeD -\rhoe}{Y_T} \leq O(\Delta a).
$$
\end{lemma}
\begin{proof}
Using the method of characteristics one gets~:
 $$
 \begin{aligned}
&    \delta  \orhoed{n+\nud}{j+\nud} + \Delta a \ztepsd{n+1}{j+1}\orhoed{n+1}{j+1} =:e^{n+1}_{j+1} =
\\
& = \frac{1}{|C^n_j|}\int_{C^n_j} 
\left(
(1+\ztepsd{n+1}{j+1} \Delta a)
\exp\left(-\int_0^{\Delta a} \zteps(\bfx,a+s,t+\e s)ds\right)-1
\right)
\rhoe(\bfx,a,t) da dt \\
& = \frac{1}{|C^n_j|}\int_{C^n_j} 
\left(\left(\ztepsd{n+1}{j+1}-\zteps(\bfx,a,t)\right) \Delta a + O(\Delta a^2)\right)\rhoe(\xat) da dt  \leq C\nrm{\zteps}{W^{1,\infty}_{\xat}} \Delta a^2 \orhoed{n}{j},
 \end{aligned}
$$
for all $j\geq 0$. 
In a similar fashion one derives for $n\geq 1$
$$
e^n_0 := | (1+\Delta a \ztepsd{n}{0})\orhoed{n}{0}-\bepsd{n}(1-\omuzed{n})| \leq \Delta a \nrm{\beps }{W^{1,\infty}_{\bfx,t}} (1 + \nrm{\zteps}{W^{1,\infty}_{\xat }}),
$$
while if $n=0$, 
$$
\begin{aligned}
e^0_0 =| \orhoed{0}{0}(1+\Delta a \ztepsd{0}{0})-\bepsd{0}(1-\omuzed{0})| \leq \nrm{\rhoe}{L^\infty_{\xat}} (1+\Delta a \ztmax)+\bmax \leq C_0.
\end{aligned}
$$ 
Setting $E^n = \Delta a \sum_{j\in\N} |\rhoed{n}{j}-\orhoed{n}{j}|$,
the previous estimates  give for $n\geq 1$
$$
E^{n+1} \leq \alpha ( E^n+ C_2 \Delta a^2 )
$$
and 
$$
E^0 \leq \alpha (E^{-1} + C_1\Delta a + C_2 \Delta a^2) \leq C_3 \Delta a,
$$
where $\alpha :=1/(1+\Delta a\ztmin)$ and by definition $E^{-1}=0$.
Combining these estimates leads to
$$
E^{n+1} \leq  C \left(  \Delta a  + \frac{\alpha}{1-\alpha} \Delta a^2 \right) \leq   C \Delta a,
$$
which gives the first result.
Using similar arguments as in Lemma \ref{lem.err.bv}, one can show that
$$
\nrm{\orhoeD-\rhoe}{Y_T} \lesssim \Delta a \nrm{\rhoe}{U_T},
$$
which gives the second result.
\end{proof}

\begin{theorem}\label{thm.cvg.dicrete.rhoe}
 Under hypotheses \ref{hypo.data} and \ref{hypo.data.deux}, one has
 $$
(1+a) \rhoeD \to (1+a) \rhoe 
 $$
 strongly in $Y_T$ when $\Delta a$ goes to zero for $\e$ fixed.
\end{theorem}
\subsection{Existence, uniqueness and stability of the discrete solution $\zepsD$} 

Existence of minimizers relies on the convexity of the Dirichlet norm and is standard as the few properties 
listed below (see for instance Lemma 1 and 2, p. 973 \cite{OeSch}).
\begin{theorem}\label{thm.exists.minimis}
 Under hypotheses \ref{hypo.data}, \ref{hypo.data.deux}, \ref{hypo.data.trois}, 
 for every $n\geq0$ there exists a minimizer $\zepsd{n}\in\pmr$ of \eqref{eq.disc.min}, {\em i.e.} there exists a 
 minimizing subsequence $(\zepsd{n,k})_{k\in\N}$ s.t. as $k\to\infty$,
\begin{enumerate}[1)]
 \item $\zepsd{n,k}\rightharpoonup \zepsd{n}$ weak in $\bho$,
 \item $\zepsd{n,k}\to\zepsd{n}$ strong in $\blo$,
 \item $\zepsd{n,k}\to\zepsd{n}$ a.e. $\bfx\in\om$,
 \item $\zepsd{n}\in\pmr$ and thus $\zepsd{n}\neq0$.
\end{enumerate}
\end{theorem}

A way to insure convergence, when $\e$ or $\Da$ go to zero, is to obtain some control  on a discrete time derivative of $\zepsD$, typically 
an $\bL^2_{\bfx,t}$-bound is obtained in the case of a classical gradient flow  directly from 
the minimization principle (cf Appendix in \cite{Oel.11} and references therein).
Here the result is less immediate~: first, in the next lemma,  we obtain a dissipation term in the energy 
estimates. These estimates provide a uniform bound on the dissipation term. 
It then appears as a source term in a closed equation \eqref{eq.delta.z.npud}, on 
$\delta \zepsd{n+\nud}$ that  finally provides these key estimates (cf. Proposition \ref{prop.time.compactness}).

\begin{lemma}\label{lem.nrj}
 If $(\rhoed{n}{i})_{(i,n)\in\N^2}$ and $(\zepsd{n})_{n\in\N}$, are defined as above, one
 has :
\begin{equation}\label{eq.energy.dscrt}
  \cE_{n+1}(\zepsd{n+1}) + \sum_{m=1}^n  \Delta t \cD_m
  \leq \cE_0(\zepsd{0})  \leq C,\quad \forall n \in \N
\end{equation}
 where the dissipation term reads :
 $$
 \cD_n := \frac{\Delta a}{2} \int_\om \sum_{j\in \N} \left|\vepsd{n}{j}\right|^2 \ztepsd{n+1}{j+1}\rhoed{n+1}{j+1} d \bfx,\quad \vepsd{n}{j} := \ue \left( \zepsd{n} - \frac{(\zepsd{n-j}+\zepsd{n-j-1})}{2}\right),
 $$
 and we denote  by  $\vepsd{n}{j}$ the discrete elongation variable
 for  $(j,n)\in\N^2$. 
 The generic constant $C$ in \eqref{eq.energy.dscrt} 
 is independent either of $\e$ or $\Da$.
\end{lemma}

\begin{proof}
 By definition of the minimization process, one has 
 $$
 \cE_{n+1}(\zepsd{n+1}) \leq \cE_{n+1}(\zepsd{n}),
 $$
 since $\zepsd{n+1}$ minimises the energy at time step $t=(n+1)\Delta t$.
 This reads
 $$
\begin{aligned}
\cE_{n+1} & (\zepsd{n+1}) \leq 
\frac{\Delta a}{4\e} \int_\om   \sum_{j=1}^{\infty} (|\zepsd{n} - \zepsd{n+1-j}|^2 + |\zepsd{n} - \zepsd{n+1-j-1}|^2 ) \rhoed{n+1}{j}  d\bfx
+ \nud \int_{\om} |  \dx \zepsd{n}  |^2 d\bfx \\
&\leq  \frac{\Delta a}{4\e} \int_\om   \sum_{j=1}^{\infty} ( |\zepsd{n} - \zepsd{n-(j-1)}|^2 + |\zepsd{n} - \zepsd{n-j}|^2 ) \left( \rhoed{n}{j-1} -\frac{\Dt}{\e} \ztepsd{n+1}{j} \rhoed{n+1}{j}\right)   d\bfx\\
&  \quad 
+ \nud \int_{\om} |  \dx \zepsd{n} |^2 d\bfx .
\end{aligned}
$$
Changing the indices in the first summation of the latter right hand side provides
$$
\begin{aligned}
{\cE_{n+1}} & {(\zepsd{n+1}) } 
\leq 
 \frac{\Delta a}{4\e} 
 \int_\om  
  \sum_{j=1}^{\infty} ( |\zepsd{n} - \zepsd{n-j}|^2 + |\zepsd{n} - \zepsd{n-j-1}|^2 )  \rhoed{n}{j}
  + \rhoed{n}{0}|\zepsd{n}-\zepsd{n-1}|^2 d \bfx - \\ 
&  -  \frac{\Delta a \Dt}{4\e^2} \int_\om   \sum_{j=1}^{\infty}  \ztepsd{n+1}{j} \rhoed{n+1}{j}  
 (|\zepsd{n} - \zepsd{n+1-j}|^2 + |\zepsd{n} - \zepsd{n-j}|^2 )
 d\bfx \, +\\
&  \quad 
+ \nud \int_{\om} |  \dx \zepsd{n} |^2 d\bfx \\
\leq & \cE_n(\zepsd{n})  
- \frac{\Delta a\Delta t}{2} \int_\om \sum_{j=0}^\infty|\vepsd{n}{j}|^2 \rhoed{n+1}{j+1}\ztepsd{n+1}{j+1} d\bfx 
= \cE_n(\zepsd{n}) - \Delta t \cD_n,
\end{aligned}
$$
for all $n\in\N$. In the last estimates we used the convexity of the square function, writing
\begin{equation}\label{eq.convex}
 \begin{aligned}
 \left| \vepsd{n}{j-1}\right|^2 & = \frac{1}{\e^2}\left|\zepsd{n} - \frac{(\zepsd{n-j}+ \zepsd{n-j+1})}{2}\right|^2 \\
 & \leq \frac{1}{2 \e^2} \left\{ |\zepsd{n}-\zepsd{n-j}|^2 + |\zepsd{n}-\zepsd{n-j+1}|^2 \right\},
\end{aligned}
\end{equation}
where $j\geq1$, while  for $j=0$, one has simply 
$\left(\vepsd{n}{0}\right)^2  \leq \left(\delta \zepsd{n-\nud} \right)^2/(2 \e^2)$. 
For $\zepsd{0}$, one has simply that
$$
\begin{aligned}
 \cE_0(\zepsd{0}) \leq & \cE_0(\zepsd{-1}) \leq  \nrm{\zepsd{-1}}{\bho}^2 
+ \frac{\Delta a}{4 \e} \int_\om \sum_{j=1}^\infty \rhoed{0}{j} \left(
\left| \zepsd{-1}-\zepsd{-j} \right|^2 + \left| \zepsd{-1}-\zepsd{-j-1} \right|^2 \right) d\bfx.
\end{aligned}
$$
Using \eqref{eq.lip.past}, one has that for almost every $\bfx \in \om$ and $j>1$
$$
| \zepsd{-1}-\zepsd{-j}| \leq \frac{C_{\zp}(\bfx)}{\Dt} \int_0^{\Delta t} | \zp(s)-\zp(s+(1-j)\Dt) | dt \leq  C_{\zp}(\bfx) \Delta t (j-1).
$$
Moreover one notices that $\rhoed{0}{j}\leq \rhoed{-1}{j-1}$ for $j\geq 1$.
Together these facts allow to give a bound  on $\cE_0(\zepsd{0})$ uniform with respect to $\e$ and $\Da$~: 
$$
\cE_0(\zepsd{0}) \leq \e \nrm{(1+a)^2 \rhoi}{L^1(\rr;L^\infty(\om))} \nrm{C_{\zp}}{L^2 (\om)}^2 + \nrm{\zepsd{-1}}{\bho}^2.
$$
\end{proof}


For a.e. $\bfx \in \om$, we denote by 
$
\clen(\bfx) := 
 \Delta a 
\sum_{j \in\N} \rhoed{n}{j} \vepsd{n}{j}. 
$
\begin{lemma}
 For every time $t^n=n\Delta t$, $\zepsd{n}$ solves :
 \begin{equation}\label{eq.eul.lag.disc}
  ( \clen, \bfv ) + \int_{\om} \lamed{n} \, \zepsd{n}\cdot \bfv \,dx + ( \dx \zepsd{n},\dx \bfv)=0,
\end{equation}
 for all $\bfv \in \bho$, and $\lamed{n}(\bfx):= - \cL_\e^n \cdot \zepsd{n} - | \dx \zepsd{n} |^2$,
 is a $L^1(\om)$ function.
\end{lemma}

\begin{proof}
We take $\bfv \in \bho$, and  set 
$$
\bfv(\tau) := \frac{\zepsd{n}+ \tau\bfv}{\left| \zepsd{n}+ \tau\bfv\right| },
$$
because $\zepsd{n} \in \cA$ for a $\tau$ small enough $\left| \zepsd{n}+ \tau\bfv\right|$ 
is strictly positive and bounded, thus on this interval $\bfv(\tau) \in \cA$.
As $\zepsd{n}$ minimizes $\cE_n$,   $i(\tau):=\cE_n(\bfv(\tau))$ admits a minimum
in $\tau=0$. This leads to $i'(0)=0$, as $\partial_\tau \bfv(0)= (\id{d} -\zepsd{n}\otimes\zepsd{n})\bfv$, this gives
$$
\left( \cL_\e^n,(\id{d}-\zepsd{n}\otimes\zepsd{n})\bfv\right) 
+ \left( \dx \zepsd{n},\dx ((\id{d}-\zepsd{n}\otimes\zepsd{n})\bfv\right) =0,
$$
where the parentheses denote the $\bL^2(\om)$ scalar product and $\id{d}$ the identity matrix in $\RR^d$. 
 As $\dx \zepsd{n} \cdot \zepsd{n} = 0$ for almost every $\bfx \in\om$, 
the previous expression transforms into 
$$
\left( \cL_\e^n,(\id{d}-\zepsd{n}\otimes\zepsd{n})\bfv\right) 
+ \left( \dx \zepsd{n},\dx \bfv \right) 
- \int_\om  \zepsd{n}\cdot \bfv \, | \dx \zepsd{n} |^2d\bfx =0,
$$
for all $\bfv \in \bho$.
Denoting $\lamed{n}:= - \cL_\e^n \cdot \zepsd{n} - | \dx \zepsd{n} |^2$,
it is a Lagrange multiplier associated to the constraint. Thanks 
to Theorem \ref{thm.exists.minimis} and Lemma $\ref{lem.nrj}$, 
$\lambda^n \in L^1(\om)$.
Thus, \eqref{eq.eul.lag.disc} together with the constraint $|\zepsd{n}|=1$ is the Euler-Lagrange system associated to 
the discrete minimization problem \eqref{eq.disc.min}.
\end{proof}

\begin{proposition}\label{proposition.lag.multi}
 Under the previous hypotheses, one has the estimate
 $$
 \forall n \in \{ 0,\dots, N\}, 
 \quad \nrm{\lamed{n}}{L^1(\om)} \leq C,
 $$
 where the constant is uniformly bounded with respect to  $\e$,$\Da$ and $\Dt$.
\end{proposition}
\begin{proof}
As $\zepsd{n} \in L^\infty_t \bH_x^1$ uniformly in $\e$, it is already clear 
that $| \dx \zepsd{n}|^2$ belongs to $L^\infty_t L^1_\bfx$.
It remains to estimate 
$\nrm{\cled{n}\cdot\zepsd{n}}{L^1(\om)}$. Since $\zepsd{n-j}\in \cA$  for all $j \in\N$ 
(this statement uses 
the first assumption in hypotheses \ref{hypo.data.trois}, in the case when $n-j<0$), a simple computation gives that,
for every $\bfx \in \om$, 
$$
(\zepsd{n}-\zepsd{n-j})\cdot \zepsd{n} = \nud \left( \zepsd{n}-\zepsd{n-j}\right)^2 \geq 0 , \quad \forall j \in \N.
$$
This in turn  suggests that
$$
\begin{aligned}
 \frac{1}{4\e} & \int_\om \left\{ \sum_{j=1}^\infty ( (\zepsd{n}-\zepsd{n-j})^2 + (\zepsd{n}-\zepsd{n-j-1})^2 ) \rhoed{n}{j} + (\zepsd{n}-\zepsd{n-1})^2\rhoed{n}{0}\right\} \Delta a d\bfx \\
& = \int_{\om} \cled{n} \cdot \zepsd{n} d\bfx = \int_{\om} | \cled{n} \cdot \zepsd{n} | d\bfx.
\end{aligned}
$$
By Lemma \ref{lem.nrj}, the first term is bounded for any $n\geq 0$. Thanks to the definition of $\lamd{n}$, the claim follows.
\end{proof}

\begin{remark}
Proposition \ref{proposition.lag.multi}  shows as well that the energy minimization procedure provides a $L^\infty_t L^1_\bfx$ bound,  uniform in $\e$, 
on the Lagrange multiplier $\lameD$. Direct use of the energy  estimates from Lemma \ref{lem.nrj}  and  Jensen's inequality give 
$
\nrm{\cled{n}}{L^2(\om)} \lesssim \sqrt{ \cE_n(\zepsd{n})/ \e}  \lesssim \e^{-1/2}
$
which provides only
$$
\nrm{\lamed{n}}{L^1(\om)} \leq C \nrm{\cled{n}}{L^2(\om)} + \nrm{\dx \zepsd{n}}{L^2(\om)}^2 \leq O(\e^{-\nud}).
$$
\end{remark}
\begin{remark}
 The previous result shows that the delay operator $\cled{n}$ points out of the unit sphere since
by convexity of the square function, $\cled{n}\cdot \zepsd{n} > 0$  for a.e. $\bfx \in \om$. 
In the next proposition, we show that the scalar product is  of order $\e$ with respect to the $L^1_{\bfx,t}$ norm, 
which makes sense. Indeed, when $\e$ is small,  $\cled{n}$ approximates
$\muoz \dt \zz$ which is tangent to the sphere,
and thus orthogonal to $\zz$.
\end{remark}
\begin{proposition}\label{prop.decay}
 Under hypotheses \ref{hypo.data}, \ref{hypo.data.deux} and \ref{hypo.data.trois}, one can also show that
 $$
 \nrm{\cleD\cdot\zepsD}{L^1_{\bfx,t}}=\Delta t \sum_{n \in \N} \int_\om \cled{n} \cdot \zepsd{n} d\bfx \leq \e C,
 $$ 
 where the constant does not depend on $\e$.
\end{proposition}
\begin{proof}
 Using again the same idea as in the previous proof, one writes :
 $$ 
\begin{aligned}
  \Delta t \sum_{n \in \N} & \int_\om \cled{n} \cdot \zepsd{n} d\bfx & 
   =  \Delta t \frac{\Delta a}{\e} \sum_{n \in \N} \int_\om  \sum_{j=1}^{\infty}  \frac{\ztepsd{n}{j}}{\ztepsd{n}{j}} \rhoed{n}{j}  
 (|\zepsd{n} - \zepsd{n-j}|^2 + |\zepsd{n} - \zepsd{n-j-1}|^2 )
 d\bfx 
\leq \frac{\e C}{\ztmin},
\end{aligned}
$$
the latter estimate coming from the dissipation term in the proof of Lemma \ref{lem.nrj}.
 \end{proof}

Here we show one of the key estimates of the paper.
\begin{proposition}\label{prop.time.compactness}
 Under hypotheses above, and for $\Delta t$ small enough, one has :
 $$
 \sum_{n=1}^{N} \Delta t \left\{ \nrm{ \frac{ \zepsd{n+1}-\zepsd{n}}{\Delta t}}{\bL^2(\om)}^2 
 + \e \nrm{ \frac{\dx \zepsd{n+1}-\dx \zepsd{n}}{\Delta t}}{\bL^2(\om)}^2 \right\} 
 \leq C, 
 $$
where the constant does not depend neither on $\e$ nor on $\Delta t$.
\end{proposition}
\begin{proof}
Recalling the definition of $\vepsd{n}{j}$ one checks easily that 
 $$
\e  \delta \vepsd{n+\nud}{j}
   + \Delta t  \frac{ \delta \vepsd{n}{j-\nud}}{\Delta a} 
 = \delta \zepsd{n+\nud}\quad \forall j\geq 1,
 $$
 while $\vepsd{n}{0}=\delta \zepsd{n-\nud}/(2 \e)$. 
Equivalently, because of the specific CFL condition, $\vepsd{n+1}{j+1}=\vepsd{n}{j}+ \delta \zepsd{n+\nud} /\e$ for all $j\geq 1$. 
Setting $\tepsd{n}{j} = \rhoed{n}{j}\vepsd{n}{j}$ for $j\in\N$, one obtains using \eqref{eq.def.scheme.rho.int}~:
 $$
\e \delta \tepsd{n+\nud}{j} +\Delta t \frac{\delta \tepsd{n}{j-\nud}}{\Delta a} + \Delta t \ztepsd{n+1}{j} \tepsd{n+1}{j} = \rhoed{n}{j-1}\delta \zepsd{n+\nud},
 $$
 which, summing over  $j\in \N^*$, gives
 $$
 \e \sum_{j\geq1} \delta\tepsd{n+\nud}{j} \Delta a - \e \Delta a  \tepsd{n}{0}+ \Delta t \sum_{j \geq 1} \ztepsd{n+1}{j}\tepsd{n+1}{j} \Delta a= \muzed{n} \delta \zepsd{n+\nud}.
$$
By definition, 
$$
\begin{aligned}
\e \Da \tepsd{n+1}{0}\equiv  \e \Da \rhoed{n+1}{0} \vepsd{n+1}{0} & = \e \Da \left( \rhoed{n+1}{b} \vepsd{n+1}{0} - \Delta a \ztepsd{n+1}{0}\rhoed{n+1}{0} \vepsd{n+1}{0} \right) \\
& =  \e \Da\left( \rhoed{n+1}{b}  \frac{\delta\zepsd{n+\nud}}{2 \e} - \frac{\Dt}{\e} \ztepsd{n+1}{0}\tepsd{n+1}{0} \right).
\end{aligned}
$$
Adding both equations  gives :
$$
 \e \delta \cled{n+\nud} + \Delta t \sum_{j \in \N} \ztepsd{n+1}{j}\tepsd{n+1}{j} \Delta a= \left( \muzed{n} + \frac{\Delta a}{2} \rhoed{n+1}{b}\right)\delta \zepsd{n+\nud},
 $$
 since $\sum_{j\in\N} \tepsd{n}{j} \Delta a = \sum_{j\in\N} \rhoed{n}{j} \vepsd{n}{j} \Delta a = \cled{n}$. 
 Now we make  the discrete difference of \eqref{eq.eul.lag.disc} between steps $n+1$ and $n$, 
 in order to express $\delta \cled{n+\nud}$ as a function
 of $\delta \zepsd{n+\nud}$. This reads :
 $$
 (\delta \cled{n+\nud},\bfv) +\left(\dx \left( \delta \zepsd{n+\nud}\right) ,\dx \bfv\right) +  (\delta (\lamed{}\zepsd{})^{n+\nud},\bfv) =0 
 $$
 We now close the problem solved by $\delta \zepsd{n+\nud}$~: 
\begin{equation}\label{eq.delta.z.npud}
\begin{aligned}
& \left(\left( \muzed{n} + \frac{\Delta a}{2} \rhoed{n+1}{b}\right) \delta \zepsd{n+\nud},\bfv\right)
+ \e \left( \dx \left(\delta \zepsd{n+\nud}\right),\dx \bfv\right)\\
& + \e \left\{ \int_{\om}  \lamd{n+1} \zepsd{n+1} \bfv d\bfx
-  \int_{\om} \lamd{n}\zepsd{n} \bfv d\bfx  \right\} 
=\Delta t \left(  \sum_{j \in \N} \ztepsd{n+1}{j}\tepsd{n+1}{j} \Delta a,\bfv \right).
\end{aligned}
\end{equation}
We rewrite the difference 
$$
\begin{aligned}
J^{n+\nud}(\bfv) := &\int_\om \nud \left\{ \delta \lamd{n+\nud} (\zepsd{n+1}+ \zepsd{n})+ (\lamd{n+1}+\lamd{n})(\delta \zepsd{n+\nud}) \right\} \bfv d \bfx .
\end{aligned}
$$
Applying $J^{n+\nud}$ to $\bfv=\delta \zepsd{n+\nud}$ and using that both $\zepsd{n+1}$ and $\zepsd{n}$ satisfy the constraint, reduces to :
$$
\begin{aligned}
J^{n+\nud}(\delta \zepsd{n+\nud}) := 
&\nud \int_{\om} (\lamd{n+1}+ \lamd{n}) \left| \delta \zepsd{n+\nud} \right|^2 dx,  
\end{aligned}
$$
cancelling the term containing the finite differences  $\delta \lamd{n+\nud}$.
Next we use the crucial estimates from Proposition \ref{proposition.lag.multi}, indeed :
$$
J^{n+\nud}(\delta \zepsd{n+\nud}) \leq \left(\nrm{\lamd{n}}{L^1(\om)}+\nrm{\lamd{n+1}}{L^1(\om)}\right) \nrm{\delta \zepsd{n+\nud} }{L^\infty(\om)}^2.
$$
In one space dimension, the Gagliardo-Nirenberg estimates (cf. \cite{AdFou.book}, p. 140, Theorem 5.9) provide
$$
\begin{aligned}
 \nrm{\delta \zepsd{n+\nud} }{L^\infty(\om)}^2 \leq & 
C \nrm{\delta \zepsd{n+\nud} }{H^1(\om)} \nrm{\delta \zepsd{n+\nud} }{L^2(\om)} \\
&  \leq C \left( \nrm{\delta \zepsd{n+\nud} }{L^2(\om)}^2 + 
\nrm{\dx \delta \zepsd{n+\nud} }{L^2(\om)} \nrm{\delta \zepsd{n+\nud} }{L^2(\om)}  \right)\\
& \leq C \left( \e^{-\nud} \nrm{\delta \zepsd{n+\nud} }{L^2(\om)}^2 + \e^\nud \nrm{\dx \delta \zepsd{n+\nud} }{L^2(\om)}^2 \right).
\end{aligned}
$$
Thus setting $\bfv=\delta \zepsd{n+\nud}$ in the weak formulation above 
and because there is an $\e$ in front of $J^{n+\nud}$ in \eqref{eq.delta.z.npud} one writes finally :
$$
\begin{aligned}
 (\mumin - 2C \sqrt{\e}) &\nrm{\delta \zepsd{n+\nud}}{L^2(\om)}^2 + ( \e  - C \e^{\frac{3}{2}} ) \nrm{\dx \delta \zepsd{n+\nud}}{L^2(\om)}^2 \\
 & \leq 
\Delta t \nrm{\sum_{j \in \N} \ztepsd{n+1}{j}\tepsd{n+1}{j} \Delta a}{L^2(\om)}\nrm{\delta \zepsd{n+\nud}}{L^2(\om)}.
\end{aligned}
$$
Using Young's inequality on the right hand side above,  for $\e$ small enough,  one has :
$$
\begin{aligned}
\frac{1}{\Delta t} \sum_{n=0}^N \nrm{\delta \zepsd{n+\nud}}{L^2(\om)}^2 & \lesssim \sum_{n=0}^N  \Delta t \nrm{\sum_{j \in \N} \ztepsd{n+1}{j}\tepsd{n+1}{j} \Delta a}{L^2(\om)}^2 \\
& \lesssim  \Delta t \sum_{n=0}^N \int_\om \sum_{j\in \N} \rhoed{n}{j} (\vepsd{n}{j})^2 \Delta a = \Delta t \sum_{n=0}^N \cD_n \leq C. 
\end{aligned}
$$
\end{proof}
The previous argument provides uniqueness as well~:
\begin{proposition}
 Under hypotheses \ref{hypo.data}, \ref{hypo.data.deux} and \ref{hypo.data.trois}, there exists a unique solution $\zepsD \in L^\infty((0,T);\bho)\cap H^1((0,T);\bL^2(\om))$ solving \eqref{eq.eul.lag.disc}.
\end{proposition}
\begin{proof}
 We use induction arguments to show the claim. We suppose that there exists two solutions ${\mathbf z}_{\e,\Delta,i}$ for $i\in\{1,2\}$.
 We denote by $\delta {\mathbf z}^k := {\mathbf Z}^k_{\e,2}-{\mathbf Z}^k_{\e,1}$, and we write the equation it satisfies for $k=0$~:
 $$
 \left( \left(\sum_{j=1}^\infty\rhoed{0}{j} + \rhoed{0}{0}/2\right) \Delta a\; \delta {\mathbf z}^0,\bfv\right)+ \e (\dx \delta {\mathbf z}^0,\bfv)+ \e (\delta \lamd{0}{\mathbf z}_{\e,2}+ \lambda_{\e,1}^0 \delta {\mathbf z}^0,\bfv)=0.
 $$
 Thus choosing $\bfv = \delta {\mathbf z}^0$ and using the same arguments as above implies that $\delta {\mathbf z}^0=0$.
 We suppose at this point that  $\delta {\mathbf z}^k=0$ for $k\leq n$. Then a careful decomposition of ${\mathbf  {\mathcal L}}^n_{\e,\Delta, 2} - {\mathbf  {\mathcal L}}^n_{\e,\Delta, 1}$ leads to 
 $$
 \left( \left\{ \sum_{j=1}^\infty\rhoed{n+1}{j} + \rhoed{n+1}{0}/2 \right\} \Delta a \delta {\mathbf z}^{n+1},\bfv\right)+ \e (\dx \delta {\mathbf z}^{n+1},\bfv)+ \e (\delta \lamd{n+1}{\mathbf z}_{\e,2}+ \lambda_{\e,1}^{n+1} \delta {\mathbf z}^{n+1},\bfv)=0,
 $$
which again, thanks to the lower bound $\mumin$ established in Lemma \eqref{lem.muze.bounds}, shows that
$$
\left(\frac{\mumin}{2}-\sqrt{\e}c_1\right)\nrm{ \delta {\mathbf z}^{n+1}}{\bL^2(\om)}^2 + \e( 1 -c_2 \sqrt{\e}) \nrm{\dx  \delta {\mathbf z}^{n+1}}{\bL^2(\om)}^2 \leq 0,
$$
proving the claim for $\e$ small enough and $k=n+1$. This ends the proof since $ {\mathbf z}_{\e,\Delta,2}= {\mathbf z}_{\e,\Delta,1}$.
\end{proof}

\begin{proposition}\label{prop.cvg.cz}
 Under hypotheses \ref{hypo.data}, \ref{hypo.data.deux} and \ref{hypo.data.trois}, $\tizepsD$, the piecewise linear interpolation of $(\zepsd{n})_{n\in\Z}$ satisfies
 $$
\tizepsD \in C^{0,\frac{(1-\gamma)}{4}}([0,T];C^{0,\gamma}(\ov{\om}))
$$
for every $\gamma \in (0,1)$, the bound is uniform with respect to $\Delta t$ and $\e$. 
Thus
$\tizepsD$ converges strongly in $C^0(\ov{\om}\times[0,T])$ when $\Delta t$ goes to zero. 
Moreover,  $\zepsD$ converges strongly in $L^\infty((0,T);C(\ov{\om}))$. 
\end{proposition}

\begin{proof}
 Thanks to Lemma \ref{lem.nrj}, $\tizepsD$ belongs to $L^\infty_t {\mathbf H}^1_x$ uniformly with respect to  $\e$, which shows weak-$\star$ convergence in this space.
 Weak convergence in $H^1_t {\mathbf L}^2_x$ follows from Proposition \ref{prop.time.compactness}. The interpolation inequality
 $$
 \nrm{u}{C^{0,\gamma}(\ov{\om})} \leq c \nrm{u}{H^{(\gamma+1)/2}(\om)} \leq c \nrm{u}{L^2(\om)}^{\frac{(1-\gamma)}{2}} \nrm{u}{H^1(\om)}^{\frac{(1+\gamma)}{2}}
 $$
holds for every $u \in H^1(\om)$  and for every $\gamma \in (0,1)$. Combined with the $L^\infty_t {\mathbf H}^1_x$ bound provided by Lemma \ref{lem.nrj}, this leads to :
$$
\nrm{\tizepsD(t_2)-\tizepsD(t_1)}{C^{0,\gamma}(\ov{\om})} \leq c (t_2-t_1)^{\frac{(1-\gamma)}{4}}. 
$$
We complete the convergence proof for $\tizepsD$ by an application of the Ascoli-Arzela theorem.
\end{proof}

\begin{corollary}\label{coro.conv.zeps}
 Under the previous hypotheses, the same result can be derived for $\zeps := \lim_{\Delta t \to 0} \zepsD$, {\em i.e.}
  $$
\zeps \in C^{0,\frac{(1-\gamma)}{4}}([0,T];C^{0,\gamma}(\ov{\om}))
$$
for every $\gamma \in (0,1)$, the bound is uniform with respect to  $\e$. 
This implies that 
$\zeps$ converges to $\zz$ strongly in $C^0(\ov{\om}\times[0,T])$ when $\e$ goes to zero. 
\end{corollary}
\begin{proof}
 Considering $\tizepsD$, the  piecewise continuous function in time, $\dt \tizepsD$ is bounded in $\bL^2_{\bfx,t}$ uniformly with respect to $\e$,
 thus $\dt \tizepsD\wcvg \dt \zeps$ weakly in $\bL^2_{\bfx,t}$ and one has that
 $$
\nrm{ \dt \zeps}{\bL^2_{\bfx,t}} \leq \liminf_{\Delta \to 0} \nrm{\dt \tizepsD }{\bL^2_{\bfx,t}} =  \liminf_{\Delta \to 0} \left( \Delta t \sum_{n\in\N} \nrm{\delta \zepsd{n+\nud} /\Delta t}{\bL^2(\om)}^2 \right)^\nud.
$$
A similar argument provides an $L^\infty_t \bH^1_\bfx$ bound for $\zeps$. 
One can then follow again the same steps as in the proof of Proposition \ref{prop.cvg.cz}.
\end{proof}

\section{Convergence when $\e$ is fixed and $\Delta a$ goes to 0.}\label{sec.conv.eps.fixed}

\newcommand{\bsvarphi}{{\boldsymbol \varphi}}
		
Next, we consider the convergence of $\cleD(\bfx,t) := \sum_{n=0}^N \chiu{(n,n+1)\Delta t}(t) \cled{n}(\bfx)$.
\begin{proposition}\label{prop.cled}
 Under hypotheses \ref{hypo.data}, \ref{hypo.data.deux} and \ref{hypo.data.trois}, for every fixed  $\e>0$,  
 the discrete delay term converges to the continuous limit when $\Delta a$ goes to zero, {\em i.e.}
 $$
 \int_0^T \int_\om \cleD (\bfx,t) \varphiD (\bfx,t) d\bfx dt \to \int_0^T \int_\om \cle (\bfx,t) \bvarphi(\bfx,t) d\bfx dt 
 $$
 for all $ \bsvarphi \in C^0([0,T];L^2(\om))$ and 
 $\varphiD(\bfx,t) := \sum_{n=0}^N \chiu{(n,n+1)\Delta t}(t) \varphin(\bfx)$ where
 $\varphin(\bfx) := $ $\int_{n\Delta t}^{(n+1)\Delta t}$ $ {\boldsymbol \varphi}(\bfx,t) dt / \Delta t$.
\end{proposition}

\begin{proof}
In what follows the terms that we handle are integrable on the domain $\om\times\rr\times (0,T)$
so the systematic use of Fubini's Theorem is implicitly assumed and we freely commute integrals with respect to  space, age and time.
  We set $I_\Delta := \int_\om \int_0^T \cleD(\bfx,t) $ $ \varphiD(\bfx,t) dt d\bfx$ that we split in two parts~:
$$
\begin{aligned}
 I_\Delta :=&  \ue \int_\om \int_0^T \int_\rr \rhoeD(\bfx,a,t) \zepsD(\bfx,t) \cdot \varphiD(\bfx,t) da dt d\bfx \\
& -   \ue \int_\om   \Delta a \sum_{n=0}^{n} \int_{n\Delta t}^{(n+1)\Delta t} \sum_{j=0}^\infty \left(\frac{\zepsd{n-j}+\zepsd{n-j-1}}{2}\right)  \cdot  \varphin  \rhoed{n}{j}  dt d\bfx
 =: \ue \left( I_{1,\Delta} - I_{2,\Delta}\right).
\end{aligned} 
$$
By Lemma \ref{lem.muze.bounds},  $\mueD$ is uniformly bounded with respect to  $\e$, $\Delta t$ and $\Delta a$, $\mueD\wscvg \muze$ in the weak-$\star$ topology in $L^\infty((0,T)\times\om)$. Moreover since $L^2((0,T);L^2(\om))$ is a separable space, 
the step functions in time with values in $L^2(\om)$ are dense.  Thus $\varphiD$ tends to $\bvarphi$  strongly in $\bL^2_{\bfx,t}$ and  the product $\zepsD  \varphiD$ converges strongly in 
$L^1(\om\times(0,T))$.  All this gives :
$$
I_{1,\Delta}:=\int_0^T \int_\om \mueD \zepsD   \cdot \varphiD d \bfx dt \to \int_0^T \int_\om  \muze \zeps \cdot \bvarphi d\bfx dt.
$$  
For the second term, one first defines 
$$
\ocnj := \{ (a,t) \in \cnj \st t>\e a + (n-j)\Delta t \}, \; \ucnj := \{ (a,t) \in \cnj \st t<\e a + (n-j)\Delta t \}, 
$$ 
and then one has :
$$
\begin{aligned}
I_{2, \Delta} = &  \int_\om \sum_{n=0}^N \sum_{j=0}^{\infty} \left\{ \int_{\ocnj}  \zepsd{n-j} \cdot \varphin  \rhoed{n}{j}  da dt 
+ \int_{\ucnj} \zepsd{n-j-1} \cdot \varphin \rhoed{n}{j}  da dt \right\} d \bfx   \\
= &  \int_\om \sum_{n=0}^N \sum_{j=0}^{\infty}\int_\cnj \zepsD(\bfx,t-\e a)\cdot  \varphiD(\bfx,t)  \rhoeD(\xat)  da dt d \bfx  \\
= & \int_\om  \int_{\rr\times (0,T)}  \zepsD(\bfx,t-\e a)  \cdot \varphiD(\bfx,t)  \rhoeD(\xat) da dt d \bfx.
\end{aligned}
$$
We consider the convergence of the term $\zepsD(\bfx,t-\e a)\varphiD(\bfx,t)$ on $\om\times\{ (a,t) \in \rr \times (0,T) \st t>\e a\}$~:
$$
\begin{aligned}
 & \left| \int_0^T \int_\om \int_0^{\tse} ( \zepsD(\bfx,t-\e a) \cdot \varphiD(\bfx,t) - \zeps(\bfx,t-\e a) \cdot{\boldsymbol  \varphi}(\bfx,t))) \rhoeD da \; d\bfx d t \right|    \leq  \\
& \leq \frac{T\sqrt{T}}{\e}\left(\nrm{\varphiD}{L^\infty_tL^2_\bfx} \nrm{\zepsD-\zeps}{L^2_{\bfx,t}} +\nrm{\zeps}{L^\infty_tL^2_\bfx}\nrm{\varphiD-{\boldsymbol\varphi}}{L^2_{\bfx,t}}\right) \sim o_{\Delta a}(1)
 \end{aligned}
$$
and thus in a similar manner, as for $I_{1,\Delta}$, one proves the convergence of
$I_{2,1,\Delta}:=\int_0^T\int_\om \int_0^\tse \rhoeD(\xat) $ $\zepsD(\bfx,t-\e a) \cdot\varphiD da d \bfx dt$.
On the other hand, on $\om\times\{ (a,t) \in \rr \times (0,T) \st t<\e a\}$, one has that~:
$$
\begin{aligned}
I_{2,2,\Delta}& := \int_0^T\int_\om \int_\tse^\infty \rhoeD(\xat) (\zepsD(\bfx,t-\e a) -\zeps(\bfx,t-\e a))\cdot \varphiD(\bfx,t) da d \bfx dt \\
= &  \ue \int_0^T\int_\om \int_{\RR_-}  \rhoeD\left(\bfx,\frac{t-\tia}{\e},t\right) (\zpD(\bfx,\tia) -\zp(\bfx,\tia))\cdot \varphiD(\bfx,t) d\tia d \bfx dt \\
= &\ue \int_0^T \sum_{n<0} \int_{n\Delta t}^{(n+1)\Delta t} \esupx  \rhoeD\left(\cdot ,\frac{t-\tia}{\e},t\right) \nrm{\zpd{n}-\zp(\cdot,\tia)}{L^2_\bfx} \nrm{\varphiD(\cdot,t)}{L^2_\bfx} d\tia dt .
 \end{aligned}
$$
A simple computation shows that if $\tia \in (n,n+1)\Delta t$ then 
$$
\nrm{\zpd{n}-\zp(\cdot,\tia)}{\bL^2_\bfx} \leq \Delta t \nrm{\dt \zp}{L^\infty_t \bL^2_\bfx},
$$
which gives then that 
$$
\begin{aligned}
|I_{2,2,\Delta} | &\leq  \Delta t \nrm{\dt \zp}{L^\infty(\RR_-;\bL^2(\om))} \nrm{\varphiD}{L^\infty_t \bL^2_\bfx} \int_{\rr\times(0,T)} \esupx \rhoeD(\xat) da dt \\
& \leq C  \Delta t \nrm{\dt \zp}{L^\infty(\RR_-;\bL^2(\om))} \nrm{\varphiD}{L^\infty_t\bL^2_\bfx} T.
\end{aligned}
$$
In a similar way, one proves, thanks to hypotheses \ref{hypo.data.trois}, that
$$
\begin{aligned}
&\left| \int_0^T\int_\om \int_\tse^\infty (\rhoeD(\xat) -\rhoe(\xat) )\zeps(\bfx,t-\e a)\cdot\varphiD(\bfx,t) da d \bfx dt \right| \\
&\leq 
\nrm{(\rhoeD-\rhoe)(1+a)}{L^1_{a,t}L^\infty_\bfx} 
\left(\nrm{\zp(\cdot,0)}{\bL^2_\bfx} + \nrm{C_{\zp}}{L^2_\bfx}\right) \nrm{\varphiD}{L^\infty_tL^2_\bfx} \sim o_{\Delta a}(1).
\end{aligned}
$$
For the last part, on $\om\times\{ (a,t) \in \rr \times (0,T) \st t<\e a\}$, one has that~:
$$
\begin{aligned}
& J_{2,3,\Delta} := \int_0^T\int_\om \int_\tse^\infty \rhoe(\xat) \zp(\bfx,t-\e a)) \cdot (\varphiD(\bfx,t) -{\boldsymbol\varphi}(\bfx,t) )da d \bfx dt \\
 &\leq  \sup_{t \in (0,T)}  \int_\rr   (1+a)\esupx 
 \rhoe (\xat) da \left(\nrm{\zp(\cdot,0)}{\bL^2_\bfx} +\nrm{C_{\zp}}{L^2_\bfx} \right) \nrm{\varphiD-{\boldsymbol\varphi}}{L^1_t\bL^2_\bfx} \\
& \quad \sim  o_{\Delta a}(1),
 \end{aligned}
$$
which proves that
$$
\int_0^T\int_\om \int_\rr \rhoeD(\xat)\zepsD(t-\e a)\cdot\varphiD(x,t) -\rhoe(\xat)\zeps(t-\e a)\cdot\bvarphi(x,t)  \dxat \to 0
$$
and ends the proof.
\end{proof}

\begin{theorem}\label{thm.lim.delta}
Under  hypotheses above, there exists a unique $\zeps \in H^1((0,T);\bL^2(\om))$ $\cap $ $L^\infty((0,T);$ $\bho)$ solving, for almost every $t\in(0,T)$, 
\begin{equation}\label{eq.eul.lag.cont}
(\cle(\cdot,t) 
,\bfv ) +(\dx \zeps(\cdot,t),\dx \bfv) + \int_\om \lame(\bfx,t) \zeps(\bfx,t) \cdot \bfv(\bfx) d\bfx    =0,\quad \forall v \in \bho,
\end{equation}
where the brackets denote the $L^2(\om)$ scalar product and the Lagrange multiplier $\lame = -\cle\cdot\zeps-|\dx \zeps|^2$ is an $L^\infty((0,T);L^1(\om))$ function uniformly with respect to $\e$.
Moreover,  for  every   $(\bfx,t) \in \om\times[0,T]$,  $|\zeps(\bfx,t)|=1$.
\end{theorem}

%

\begin{proof}
By   Proposition \ref{prop.cled}, the convergence of $\cleD$ is proved.
  Since $\dx \zepsD \in L^\infty_t \bL^2_\bfx$ uniformly with respect to  $\e,\Delta a$ and $\Delta t$, one has 
$$
\int_0^T (\dx \zepsD,\dx \varphiD) dt \to \int_0^T (\dx \zeps,\dx{\boldsymbol \varphi})  dt ,
$$
where again $ {\boldsymbol\varphi} \in C^0([0,T];\bho)$ and 
 $\varphiD(\bfx,t) := \sum_{n=0}^N \chiu{(n,n+1)\Delta t}(t) \varphin(\bfx)$ where
 $\varphin(\bfx) := $ $\int_{n\Delta t}^{(n+1)\Delta t}$ $ {\boldsymbol\varphi}(\bfx,t) dt/\Delta t$.
Since $L^\infty_t M_\bfx$ is the dual space of $L^1_t C_x$ which is separable, the bounded sets in $L^\infty_t M_\bfx$
are compact for the weak-$\star$ topology $\sigma(L^\infty_t M_\bfx,L^1_t C_x)$, we denote by $\langle\cdot,\cdot \rangle$, 
the duality brackets associated to this dual paring.
Defining $J_\e$ to be
$$
\begin{aligned}
J_\e := & \left| \int_0^T \int_\om \lambda_{\e,\Delta} \zepsD \cdot\varphiD(\bfx,t)d\bfx dt -  \langle\lame,\zeps \cdot\bvarphi\rangle \right| 
  \leq 
  \int_0^T\int_\om |\lambda_{\e,\Delta} \zepsD| | \varphiD - \bvarphi| d\bfx dt  \\
& +   \int_0^T\int_\om |\lambda_{\e,\Delta} {\boldsymbol \varphi}| | \zepsD - \zeps | d\bfx dt 
 +\left|  \int_0^T\int_\om \lambda_{\e,\Delta}  \zeps  \cdot\bvarphi d\bfx dt  -  \langle\lame,\zeps \cdot\bvarphi\rangle \right| \\
    \leq   &
\nrm{\lameD}{L^\infty_t L^1_\bfx} \nrm{\zepsD}{L^\infty_{\bfx,t}} \nrm{\varphiD-{\boldsymbol \varphi}}{L^1_t C_\bfx}
+ \nrm{\lameD}{L^\infty_t L^1_\bfx} \nrm{\varphiD}{L^\infty_{\bfx,t}} \nrm{\zepsD-\zeps}{L^\infty_{\bfx,t}} \\
&\quad  +  \left| \int_0^T \int_\om \lameD \zeps \cdot \bvarphi d\bfx dt -  \langle\lame,\zeps \cdot\bvarphi\rangle\right|.
\end{aligned}
$$
The first term tends to zero thanks to the density of valued step  functions in $L^1_tL^2_\bfx$, 
the second term is small due to the strong convergence of $\zepsD$ established above, the last one 
tends to zero thanks to the weak-$\star$ convergence of $\lameD$ in $L^\infty_t\cM_\bfx$. 
At that point, the solution pair $(\zeps,\lame)$ solves :
\begin{equation}\label{eq.weak.z.eps}
 \begin{aligned}
 \int_0^T \int_\om \cle\cdot\bvarphi+ \dx \zeps \cdot \dx \bvarphi d\bfx dt + \langle \zeps\cdot \bvarphi,\lame \rangle =0,
\end{aligned}
\end{equation}
for every  $ {\boldsymbol\varphi} \in C^0([0,T];\bho)$.
Setting ${\boldsymbol \varphi}(\bfx,t) = \zeps(\bfx,t)\theta(\bfx,t)$  with $\theta \in \D(\om\times(0,T))$, in \eqref{eq.weak.z.eps}, proves that for almost every $(\bfx,t)\in\om\times(0,T)$,
$\lame = -\cle\cdot\zeps-|\dx \zeps|^2<0$ and the right hand side is a $L^\infty_t L^1_\bfx$ function.

Taking now ${\boldsymbol \varphi}(\bfx,t)=\bfv(\bfx)\psi(t)$
for any $\bfv\in \bho$ and $\psi \in \D(0,T)$ shows that \eqref{eq.eul.lag.cont} holds a.e. $t\in(0,T)$
for any $\bfv\in\bho$.

An easy computation shows that 
$$
0 \leq 1 - |\tizepsD| \leq \sum_{n\in\N} \chiu{(n\Delta t,(n+1)\Delta t)} (t) \frac{(t- n\Delta t)}{\Delta t} \left| \delta \zepsd{n+\nud} \right| ,
$$
which gives that
$$
\nrm{1- |\tizepsD|}{L^2_{\bfx,t}}^2 \leq \Delta t \sum_{n=0 }^N \frac{1}{\Delta t} \nrm{\delta \zepsd{n+\nud}}{\bL^2_\bfx}^2 \lesssim \Delta t
$$
thanks to Proposition \ref{prop.time.compactness}.
Then a triangular inequality gives :
$$
\nrm{|\zeps| -1}{L^2_{\bfx,t}} \leq \nrm{|\zeps|-|\tizepsD|}{L^2_{\bfx,t}} + \nrm{1-|\tizepsD|}{L^2_{\bfx,t}} 
\leq \nrm{\zeps-\tizepsD}{\bL^2_{\bfx,t}}+ \nrm{1-|\tizepsD|}{L^2_{\bfx,t}}.
$$
As the right hand side is arbitrary small, the left hand side is zero. 
Thus the constraint is fulfilled a.e. in $\om\times(0,T)$.
Thanks to Corollary \ref{coro.conv.zeps},  $\zeps$ is a continuous function in time and in space, so
the result holds true everywhere in $\om\times(0,T)$.

In order to prove uniqueness we assume that there exists two distinct solutions $(\zeps^i)_{i\in\{1,2\}}$  sharing the same $\zp$ 
condition for negative times, 
the same kernel $\rhoe$ and both solving \eqref{eq.eul.lag.cont} for a.e. $t \in (0,T)$. 
We denote by $\hz:=\zeps^2-\zeps^1$ and it solves :
$$
\begin{aligned}
&  (\hat{\cle}(\cdot,t),\bfv)  + ( \dx \hz(\cdot,t) ,\dx \bfv ) +\\
& + \nud \int_\om \left\{ ( \lame^2(\bfx,t)+\lame^1(\bfx,t)) \hz(\bfx,t)\cdot \bfv 
+ (\lame^2(\bfx,t)-\lame^1(\bfx,t)) (\zeps^2(\bfx,t)+\zeps^1(\bfx,t))\cdot \bfv(\bfx) \right\} d\bfx = 0
\end{aligned}
$$
Setting $\bfv = \hz(t)$ one obtains thanks to the same Gagliardo-Niremberg estimates 
as in Proposition \ref{prop.time.compactness} that
$$
(\hat{\cle},\hz) + \nrm{\dx \hz}{L^2(\om)}^2 \leq \nud \left( \sum_{i\in\{1,2\}} \nrm{\lame^i}{L^1(\om)}  \right)
\left(  \e^\nud \nrm{\dx \hz}{L^2(\om)}^2 +  \e^{-\nud} \nrm{ \hz}{L^2(\om)}^2 \right) 
$$
Making the first term in the left hand side above explicit one writes :
$$
\begin{aligned}
 \hat{\cle}(t)& \cdot \hz(t) = 
\ue \left( \int_\rr  \left( \hz(\bfx,t) - \hz(\bfx,t-\e a) \right) \rhoe(\bfx,a,t) da \right) \cdot \hz(\bfx,t) \\
& \geq \frac{1}{2\e}  \left( \int_\rr  \left( | \hz(\bfx,t)|^2 - |\hz(\bfx,t-\e a) |^2 \right) \rhoe(\bfx,a,t) da \right)  \\
& = \frac{1}{2\e} \left( \muze(\bfx,t) |\hz(\bfx,t)|^2 - \int_0^\tse |\hz(\bfx,t-\e a) |^2 \rhoe(\bfx,a,t) da \right)
\end{aligned}
$$
In the above equality this contributes to obtain :
$$
\begin{aligned}
 \left( \frac{\mumin}{2} -\e^\nud C\right)&  \nrm{\hz(\cdot,t)}{L^2(\om)}^2 + (\e-C \e^\td ) \nrm{\dx \hz(\cdot,t)}{L^2(\om)}^2 \\
 & \leq 
\nud \int_0^\tse \nrm{\hz(\cdot,t-\e a)}{L^2(\om)}^2 \sup_{\bfx \in\om} \rhoe(\bfx,a,t) da 
\end{aligned}
$$
since the second term in the left hand side is positive we omit it, setting $q(t):=\nud \nrm{\hz(\cdot,t)}{L^2(\om)}^2$ and 
using that 
$\rhoe(\xat) \leq \beps(\bfx,t) \exp( - \ztmin a) $, for a.e. $\bfx \in \om$ and $t \geq \e a$, one obtains that $q$ satisfies :
$$
(\mumin-\e^\nud C/2) q(t) \leq \bmax \int_0^\tse \exp( - \ztmin a ) q(t-\e a) da
$$
which after some easy computations provide that $\int_0^t \exp(  \ztmin \tau / \e ) q(\tau) d\tau = 0$
which in turn gives that $\exp( \ztmin t ) q(t) \leq 0$  and since $q(t)$ is positive by definition this gives 
that $\nrm{\hz(\cdot,t)}{L^2(\om)} =0$ for a.e. $t$, which shows uniqueness.
\end{proof}
\begin{proposition}
 Under the previous hypotheses, one has
 $$
 \nrm{\cle\cdot\zeps}{L^1_{\bfx,t}} \leq C \e,
 $$
 where the constant is independent of $\e$.
\end{proposition}
\begin{proof}
  Using the same arguments as in Proposition \ref{prop.cled}, one shows that $\cleD\cdot\zepsD$ tends to $\cle\cdot \zeps$
  in $L^1_{\bfx,t}$ as $\Delta \to 0$. Then using the estimate established in  Proposition \ref{prop.decay}, one concludes.
\end{proof}

%
%
\section{Convergence when $\e$ goes to zero in the continuous framework}\label{sec.conv.e}

\subsection{Convergence of the population of bonds $\rhoe$ } 

For sake of conciseness we recall here the main result of Section 5.1 \cite{MiOel.4} in which 
the convergence of $\rhoe$ towards $\rhoz$ solving \eqref{eq.rho.zero} is fully 
established.

\begin{theorem}
	Under assumptions \ref{hypo.data} and \ref{hypo.data.deux}, one has 
	\begin{displaymath}
	\mathcal{H}[\hrhoe(\bfx,\cdot,t)]   \leq 
	\mathcal{H}[\rho_{\varepsilon,I}(\bfx,\cdot)-\rho_{0}(\bfx,\cdot,0)]   e^{\frac {-\ztmin  t }{\varepsilon}}   +
	\frac{2}{\ztmin} \left\{ 
	\nrm{\mathcal{R}_{\varepsilon}}{L^\infty_{\xat}}+\nrm{\cM_\e}{L^\infty_{\bfx,t}}
	\right\}
	\end{displaymath}
	with $\mathcal{R}_{\varepsilon}(\xat):=-\varepsilon  \partial_t \rho_0(\bfx,a,t)-\rho_0(\bfx,a,t)(\zteps(\bfx,a,t)-\zeta_0(\bfx,a,t))$ and
	$ \mathcal{M}_{\varepsilon}(\bfx,t):=(\beps(\bfx,t)-\beta_0(\bfx,t))$ $\left(1-\int_0^\infty \right.$ $\left.\rho_0(\bfx,a,t) \, da \right)$.  These estimates imply 
	 $\sup_{\bfx \in \Omega} |\rhoe(\bfx,a,t) - \rhoz(\bfx,a,t )|$ converges strongly in $L^1((0,T)\times \rr, (1+a) )$ which give strong convergence in $L^1(\rr\times(0,T);L^\infty(\om),(1+a))$ when $\e$ goes to zero.
	
\end{theorem}

\subsection{Study of the initial layer and convergence of continuous and descrete time derivative of $\rhoe$} 
Here we perform a preliminary analysis in order to obtain
limits when $\e$ goes to zero of the transposition of the
delay operator. For this sake we introduce the initial 
layer and show to what limit $\dt \rhoe$ converges 
in a second step (cf Theorem \ref{thm.cvg.il}). Finally we exhibit the limit to which 
 the delayed 
part transfered on $\rhoe$ of $\cL_{\e}$ {\em i.e.} 
$(\rhoe(\bfx,a,t+\e a)-\rhoe(\bfx,a,t))/\e$ tends as a Banach valued Radon measure (cf Proposition \ref{prop.lim.cke}).
\begin{proposition}\label{prop.cvg.vague}
If $f$ is in $U_T$ then its weak derivatives $\da f$ and $\dt f$ are 
in $X_T'$. One defines the corresponding duality brackets as
$$
\langle \da f, \varphi \rangle_\xtpair{T}
:= \lim_{\sigma \to 0} \int_{\cqt} \did{\sigma}{a} f  \varphi d\bfx \;da\; dt
$$
for any $\varphi \in X_T$.
\end{proposition}


\begin{proposition}\label{prop.cvg.etroite}
If $f$ is in $U_T$ then its weak derivatives $\da f$ and $\dt f$ are 
in $Z_T'$. One defines the corresponding duality brackets as
$$
\langle \da f, \varphi \rangle_{Z_T',Z_T}
:= \lim_{\sigma \to 0} \int_{\cqt} \did{\sigma}{a} f  \varphi d\bfx \;da\; dt
$$
for any $\varphi \in Z_T$.
\end{proposition}

For sake of conciseness, the proofs of these propositions are 
postponed in  \ref{sec.app.un}.  Using then these one shows : 
%

\begin{proposition}\label{prop.cvg.etroite.rho}
  Under hypotheses \ref{hypo.data} and \ref{hypo.data.deux}, the previous convergence result can be extended 
  to $Z_T'$ where $Z_T:=C_b(\oqt;L^1(\om))$. Namely for any $\varphi$ in $Z_T$, there exists a subsequence $\tau_k$ s.t.
 $$ 
\langle D^{\tau_k}_t \rhoe,\varphi\rangle  \to \langle  \dt \rhoe,\varphi\rangle   
 $$
when $k \to \infty$.
\end{proposition}

In order to identify the limit to which $\dt \rhoe$ tends when $\e$ goes to zero, 
(part of the main ingredients were presented in Proposition 3.2 p. 10,  \cite{mi.proc}, 
but the space variable was not taken in account),
we define an initial layer, as in \cite{mi.proc}. Setting $\tit = t / \e$, we look for  $\trhoz$ solution of 
\begin{equation}\label{eq.micro}
\left\{
 \begin{aligned}
 &\dtit \trhoz + \da \trhoz + \ztz(\bfx,a,0) \trhoz = 0 ,& (\bfx,a,\tit) \in \Omega\times(\rr)^2, \\
 &\trhoz(\bfx,0,\tit) = - \bz(\bfx,0) \int_{\rr} \trhoz(\bfx,a,\tit) da, & \bfx \in \Omega,\; a=0,\; \tit>0,\\
 & \trhoz(\bfx,a,0) = \rhoi(\bfx,a) - \rhoz(\bfx,a,0)=:\trhoi(\bfx,a), &\bfx \in \Omega,\;  a>0,\tit=0
\end{aligned}
\right.
\end{equation}
and we define $\trhoze(\bfx,a,t) := \trhoz(\bfx,a,t/\e)$.
As in \cite{mi.proc}, we obtain at the microscopic level global existence and {\em a priori} bounds~:
\begin{theorem}\label{thm.micro}
 Under hypotheses \ref{hypo.data} and \ref{hypo.data.deux}, there exists a unique solution $\trhoz$ belonging to $C^0(\rr;L^1_a(\rr;L^\infty_\bfx(\om))$ $\cap$  $L^\infty(\om\times\rr\times\rr)$. Moreover, one has
 \begin{equation}\label{eq.est.decrease.micro}
 \begin{aligned}
 \int_\rr \esupx | \trhoz(\bfx,a,\tit) | da \lesssim \exp( - \ztmin \tit ),\quad \forall \tit \in \rr,
\end{aligned}
\end{equation}
\begin{equation}\label{eq.est.apriori.micro}
 \sup_{\tau \in (0,\tau_0)} \nrm{ (1+\tit) \tet \trhoz }{L^1(\rr\times\rr;L^\infty(\om))} < C
\end{equation}
 and there exists a subsequence s.t. $D^{\tau_k}_{\tit} \trhoz \wscvg \partial_{\tit} \trhoz$ weak-$*$ in the
 $\sigma(Z_\infty',Z_\infty)$ topology.
\end{theorem}
\begin{proof}
 The proof of the existence and uniqueness part is easy and follows the same ideas as in 
 \cite{MiOel.1,mi.proc} where one shall only manage the $\bfx$ dependence in addition.
  {\em A priori} estimates on $\tet \trhoz$ are  obtained as in Theorem 2.2 p. 6 \cite{mi.proc}.
 The last part follows the same ideas as in Propositions \ref{prop.cvg.vague} and \ref{prop.cvg.etroite} below.
\end{proof}
\begin{corollary}\label{coro.scale}
 Under the same hypotheses, one has the scaling 
 $$
 \langle \dt \trhoze,\varphi \rangle_{Z_T',Z_T} = \langle\dt \trhoz,\varphi(\cdot,\cdot,\e\cdot)\rangle_{Z_{T/\e}',Z_{T/\e}}.
 $$
\end{corollary}
\begin{proof}
 We start from the change of variable $\tit = t/\e$, which gives
 $$
 \int_{\om\times\rr\times (0,T)} \tet \trhoze \varphi(\bfx,a,t) d\bfx da dt = \int_{\om\times\rr\times (0,T/\e)} \D^{\ti{\tau}}_t \trhoz(\bfx,a,\tit)  \varphi(\bfx,a,\e \tit) d\bfx da d\tit, 
 $$
 where $\ti{\tau}= \tau/ \e$, then the right hand side (resp. left hand side) converges up to a subsequence to  the right hand side (resp. left hand side) of the claim
 by the same arguments as in Propositions \ref{prop.cvg.vague} and \ref{prop.cvg.etroite}.
  \end{proof}
\begin{theorem}\label{thm.cvg.il.micro}
  Under hypotheses \ref{hypo.data} and \ref{hypo.data.deux}, one has for any $\varphi \in C_b(\rr; L^1_\bfx(\om))$, 
  $$
 \lim_{\e \to 0} \; \langle \dt \trhoze,\varphi \rangle_\ztpair{T} =- \int_{\om\times\rr} \varphi(\bfx,a) \left(\rhoi(\bfx,a) -\rhoz(\bfx,a,0)\right)da d\bfx,
 $$
 and we underline that here $\varphi$ does not depend on time.
\end{theorem}
\begin{proof}
 Using  {\em a priori} estimates \eqref{eq.est.apriori.micro}, one has
 $$
\begin{aligned}
 \sup_{\tau \in (0,\tau_0)} & \int_0^{\frac{T}{\e}} \left| \tet \int_{\om\times\rr} \varphi(\bfx,a) \trhoz(\bfx,a,t) d \bfx da \right| dt 
\leq \sup_{\tau \in (0,\tau_0)} \int_0^{\frac{T}{\e}} \int_{\om\times\rr} | \varphi(\bfx,a)|  \left| \tet \trhoz(\bfx,a,t) \right| d \bfx da  dt \\
 & \leq \nrm{\varphi}{L^\infty(\rr;L^1(\om))} 
 \nrm{\trhoz}{U_\infty}< C,
\end{aligned}
 $$ 
 where we recall that $U_\infty := \BV(\rr\times\rr;L^\infty(\om))$.
This shows that $q(t):= \int_{\om\times\rr} \varphi(\bfx,a) \trhoz(\bfx,a,t) d \bfx da$ is a function of bounded variation.
Thus there exists a signed Radon measure $\nu_{\dt q}$ associated to the time derivative of $q$.
$$
\begin{aligned}
 \int_0^{T/ \e} d \nu_{\dt q} &= q(T/\e)- q(0) 
 = \int_{\om\times \rr} \varphi(\bfx,a) \trhoz(\bfx,a,T/ \e) da - \int_{\om\times\rr} \varphi(\bfx,a) \trhoi(\bfx,a) d \bfx da.
\end{aligned}
$$
Indeed, the integral $\int_0^{T/\e} d \nu_{\dt q}$ coincides with the Riemann-Stieltjes integral, thus integration by parts holds. Moreover one has that
$$
\int_0^{T/\e} D^{\tau_k}_t q(t) dt \to \int_0^{T/\e} d \nu_{\dt q}
$$
as $\tau$ (up to a subsequence) goes to zero.
On the other hand
$$
\int_{\cqt} \tet \trhoze \varphi d\bfx da dt \to \langle \dt \trhoze, \varphi \rangle_\ztpair{T}
$$
in the weak-$\star$ topology $\sigma(Z_T',Z_T)$ (as in the proof of Proposition \ref{prop.cvg.etroite}).
Because 
$$
\int_{\cqt} \tet \trhoze (\bfx,a,t) \varphi (\bfx,a) d\bfx da dt 
= \int_0^{T/\e}  D^{\tau/\e}_t q(\tit) d\tit 
$$
and the arguments above, one has finally :
$$
 \langle \dt \trhoze, \varphi \rangle_\ztpair{T} = q(T/\e)- q(0).
 $$
One concludes since $| q(T/\e) |\lesssim \exp(-\ztmin T / \e)$ thanks to \eqref{eq.est.decrease.micro}.
\end{proof}

\begin{proposition}\label{prop.err.est}
 Under assumptions \ref{hypo.data} and \ref{hypo.data.deux}, one has
 $$
\limsup_{\tau \to 0} \nrm{ \tet \left( \rhoe-\rhoz -\trhoze\right) }{Y_T} \sim o_\e(1)
 $$
 which implies that :
 $$
\lim_{\e \to 0}  \left| \langle \dt \rhoe -\dt \rhoz-\dt \trhoze,\varphi\rangle_\ztpair{T} \right| = 0
$$
for all $\varphi \in Z_T$.
\end{proposition}
\begin{proof}
 As $\bfx$ is a mute variable in the $\rhoe$ model, we first establish that :
 $$
 \esupx \int_{\qt} \left|  \tet \left( \rhoe-\rhoz -\trhoze \right)\right| da dt \sim o_\e(1)
 $$
 using exactly the same arguments as in Proposition 3.2 p.10 \cite{mi.proc}.
The method of characteristics gives, under the hypotheses above :
 $$
 \int_{\qt}  \esupx  \left|  \tet \left(\rhoe-\rhoz -\trhoze\right) \right| da dt \sim o_\e(1)
$$
as the bounds do not depend on $\tau$, the first claim follows. One writes then in the $\ztp,\zt$ duality pairing, that
$$
\begin{aligned}
  \left| \langle \dt \rhoe -\dt \rhoz-\dt \trhoze,\varphi\rangle \right|  \leq |\langle\dt \rhoe - \tet \rhoe,\varphi\rangle|+ |\langle\dt \rhoz - \tet \rhoz,\varphi\rangle|\\
+ |\langle \dt \trhoze - \tet \trhoze,\varphi\rangle|
+ |\langle\tet (\rhoe-\rhoz-\trhoze),\varphi\rangle|.
\end{aligned}
$$
Thanks to the first statement above,  for any fixed $\delta>0$ and any fixed $\varphi \in Z_T$, there exists $\e_0$ s.t. $\e<\e_0$ implies 
$$
|\langle\tet (\rhoe-\rhoz-\trhoze),\varphi\rangle| \leq \nrm{\tet(\rhoe-\rhoz-\trhoze)}{Y_T} \nrm{\varphi}{Z_T} \leq \delta/2.
$$
By Proposition \ref{prop.cvg.etroite}, there exists $\tau_0$ s.t. $\tau<\tau_0$ implies
$$
|\langle \dt \rhoe - \tet \rhoe,\varphi\rangle|
+ |\langle \dt \rhoz - \tet \rhoz,\varphi\rangle|
+ |\langle \dt \trhoze - \tet \trhoze,\varphi\rangle| \leq\delta/2,
$$
which ends the proof.
\end{proof}
\begin{proposition}\label{prop.cvg.diff}
 Under the same hypotheses, there is a limit related to the initial layer : for any $\varphi \in Z_T$,
 $$
 \lim_{\e \to 0}  \langle \dt \trhoze,\varphi(\cdot,\cdot,\cdot)-\varphi(\cdot,\cdot,0)\rangle_{\ztpair{T}}  =0.
 $$
\end{proposition}
\begin{proof}
We set $\psi(\bfx,a,t):=\varphi(\bfx,a,t)-\varphi(\bfx,a,0)$, and we use Corollary \ref{coro.scale}, giving that
$$
 \langle \dt \trhoze,\psi \rangle_{\ztpair{T}} 
 = \langle \dt \trhoz,\psi(\cdot,\cdot,\e\cdot)\rangle_{\ztpair{T/\e}}.
 $$ 
 Next we write :
 $$
\begin{aligned}
\langle \dt \trhoz,&\psi(\cdot,\cdot,\e\cdot)\rangle_{Z_{T/\e}',Z_{T/\e}} =    \langle \dt \trhoz-\tet \trhoz ,\psi(\cdot,\cdot,\e\cdot)\rangle_{Z_{T/\e}',Z_{T/\e}} \\
&+  \langle\tet \trhoz,\psi(\cdot,\cdot,\e\cdot)\rangle_{Z_{T/\e}',Z_{T/\e}} =: I_1 +I_2.
\end{aligned}
$$
We start with $I_2$ and write :
$$
 \left| I_2 \right| \leq \int_{\qt} \nrm{\psi(\cdot,a,\e t)}{L^1(\om)} \esupx | \tet \trhoz|  da dt.
 $$
Since $\esupxt  | \tet \trhoz|$ is a positive function in $L^1(\qt)$, 
there exists $\nu$ a weak-$*$ limit in $\sigma(M^1(\oqt),C^0(\oqt))$ 
of the measure $\nu_\tau$ associated to it.
Because $\nu_\tau$ is tight with respect to  $\tau$, this convergence extends to 
the weak-$*$ topology in $\sigma(C_b(\oqt)',C_b(\oqt))$.

Since $\nrm{\psi(\cdot,a,\e t)}{L^1(\om)}$ is a continuous bounded function on $\oqt$, 
converging pointwisely to 0 a.e. $(a,t)\in\qt$, there exists $\e_0$ s.t. for $\e < \e_0$,
$$
\int_{\qt} \nrm{\psi(\cdot,a,\e t)}{L^1(\om)} d\nu < \delta/3.
$$
From here until the end of the proof, $\e$ is fixed.
Thanks to the previous tight convergence result, there exists a $\tau_0$, s.t.
$$
\left| \int_{\qt}\nrm{\psi(\cdot,a,\e t)}{L^1(\om)} d\nu - \int_{\qt}\nrm{\psi(\cdot,a,\e t)}{L^1(\om)} d\nu_\tau \right| \leq \delta/3
$$
and finally there exists $\tau_1$ s.t. $\tau< \tau_1$ implies
 $$
 | I_1 | < \delta/3
 $$
 thanks to the weak-$*$ convergence in topology $\sigma(Z_{T/\e}',Z_{T/\e})$. 
Summing the three terms   ends the proof.
\end{proof}
\begin{theorem}\label{thm.cvg.il}
 Under hypotheses \ref{hypo.data} and \ref{hypo.data.deux}, one has
 $$
\begin{aligned}
  \langle \dt \rhoe,\varphi \rangle_\ztpair{T} \to 
  & \int_{\cqt} \dt \rhoz (\bfx,a,t) \varphi  (\bfx,a,t)d\bfx da dt 
 - \int_{\om\times\rr} \left( \rhoi(\bfx,a)-\rhoz(\bfx,a,0) \right) \varphi(\bfx,a,0) d\bfx da
\end{aligned}
 $$
 for every $\varphi \in \zt$.
\end{theorem}

\begin{proof}
We set 
$$
I:= \langle \dt \rhoe,\varphi \rangle 
- \int_{\cqt} \dt \rhoz \varphi d\bfx da dt 
+ \int_{\om\times\rr} \left( \rhoi(\bfx,a)-\rhoz(\bfx,a,0) \right) \varphi(\bfx,a,0) d\bfx da
$$
and split this difference adding and subtracting extra terms~:
$$
\begin{aligned}
|I |& \leq \left| \left\langle \dt \rhoe -\dt \rhoz - \dt \trhoze,\varphi\right\rangle_\ztpair{T} \right| 
+ \left| \left\langle \dt \trhoze,\varphi - \varphi(\cdot,\cdot,0)\right\rangle_\ztpair{T} \right| \\
&   \left| \left\langle \dt \trhoze, \varphi(\cdot,\cdot,0)\right\rangle_\ztpair{T} + \int_{\om\times\rr} \left( \rhoi(\bfx,a)-\rhoz(\bfx,a,0) \right) \varphi(\bfx,a,0) d\bfx da \right| =: \sum_{i=1}^3 I_i .
\end{aligned}
$$
Now for every fixed $\delta$ (small), there exists $\e_0$ s.t. $\e< \e_0$ implies 
$I_1 <\delta/3$ thanks to Proposition \ref{prop.err.est}, s.t. $I_2 <\delta/3$ thanks to Proposition \ref{prop.cvg.diff},
and s.t. $I_3<\delta/3$ thanks to Theorem \ref{thm.cvg.il.micro}, which ends the proof.~
\end{proof}

We define 
$$
{\cal K}_\e(\bfx,a,t) := a D^{\e a}_t \rhoe = \frac{ \rhoe(\bfx,a,t+\e a) - \rhoe(\bfx,a,t)}{\e} \chiu{P_\e}(a,t), \quad 
$$
where $P_\e :=\Omega\times \left\{  (a,t) \in \qt \st a<\frac{T-t}{\e} \right\}$.

\begin{theorem}\label{thm.discrete.time.derivative.rho}
Under hypotheses \ref{hypo.data} and \ref{hypo.data.deux},  ${\cal K}_\e$ solves the weak problem : for all $\psi \in Z_T$
\begin{equation}\label{eq.cke}
 \int_{\cqt} \psi \cke da dt d\bfx= \langle  \dt \rhoe ,\varphi_\e\rangle _\ztpair{T} -  \int_\cqt \varphi_\e(\xat)  (\tia D^{\e \tia}_t \zteps) \rhoe(\bfx,a,t) da dt d\bfx,
\end{equation}
where we set 
\begin{equation}\label{eq.phi.eps}
\varphi_\e(\bfx,\tia,t):=\chiu{P_\e} (\bfx,\tia,t)  \int_{\tia}^{\frac{T-t}{\e}}  \exp\left( - \int_\tia^a \zteps(\bfx,s,t+\e s) ds \right) \psi(\bfx,a,t) da. 
\end{equation}
\end{theorem}
 
\begin{proof}
 In order to express the problem solved by $\cK_\e$, we 
 regularize the data. It  gives a pointwise meaning to an approximation of $\dt \rhoe$. 
For this sake, we regularize the
boundary datum and the off-rate setting :
$$
\beps^\delta (\bfx,t) :=  (1-  \chi_\delta(t))  \beps * \omega_{\delta,t} ,\quad 
\zteps^\delta (\xat) := \zteps * (\omega_{\delta,t} \omega_{\delta,a}),
$$
where the cut-off function $\chi_\delta$ is  monotone and $C^\infty(\rr)$  s.t.
$$
\chi_\delta(a) := 
\begin{cases}
 1 & \text{ if } a<\delta, \\
 0 & \text{ if } a>2 \delta,
\end{cases}
$$
and $\omega_{\delta,a}$ and $\omega_{\delta,t}$ are  the standard mollifiers in the $a$ and $t$ variable.
For the initial condition we use as in the Appendix, the specific regularisation
of $\BV$ functions originally presented in \cite{Zie.Book,Giu.Book} for the real-valued case,
and more recently adapted  to the vector valued case in Theorem 2.21 \cite{HeiPatRen.19}.
This regularisation provides
$$
\nrm{\rhoi^\delta-\rhoi}{L^1(\rr;L^\infty(\om))} \leq \delta ,\quad \nrm{\da \rhoi^\delta}{L^1(\rr;L^\infty(\om))} \leq
\limsup_{\alpha \to 0} \nrm{ D^\alpha_a \rhoi}{L^1(\rr;L^\infty(\om))} + \delta
$$
One solves \eqref{eq.rho.eps} with initial, boundary and off-rate datum $(\rhoi^\delta, \beps^\delta,\zteps^\delta)$, 
the solution is denoted $\rhoe^\delta$.
Together with assumptions \ref{hypo.data} and \ref{hypo.data.deux}, the time derivative $\dt \rhoe^\delta$ 
solves~:
\newcommand{\rhoedt}{\dt \rhoe^\delta}
\begin{equation}\label{eq.rho.delta.eps}
\left\{
 \begin{aligned}
 & (\e \dt + \da + \zteps^\delta ) \rhoedt = \dt \zteps^\delta \rhoe^\delta, & \; a>0,\; t>0,\\
 & \rhoedt(\bfx,0,t) = \dt \beps^\delta (1- \muze^\delta ) - \beps^\delta \dt \muze^\delta & \;a=0, \; t>0, \\
 & \e \rhoedt(\bfx,a,0) =  - (\da + \zteps^\delta(\bfx,a,0 )) \rhoi^\delta & \; a>,\; t= 0, 
\end{aligned}
\right.
\end{equation}
where $\muze^\delta(\bfx,t) = \int_\rr \rhoe^\delta(\bfx,a,t)  da$.
Since the data of \eqref{eq.rho.delta.eps} is regular, 
for a fixed $\bfx \in \om$, existence results follow from Theorem 2.1 p. 488 \cite{MiOel.1}, 
and thanks to similar arguments as in Proposition \ref{prop.exist.cont.rho.eps}, one proves as well that $\rhoedt$ is a 
$C^0((0,T);L^1(\rr;L^\infty(\om)))$ function. 
One obtains {\em a priori} estimates, uniform in $\e$,  leading to $\dt \rhoe^\delta \in Y_T$.
In the same way as in Propositions \ref{prop.cvg.vague} and \ref{prop.cvg.etroite}, there is a limit $\dt \rhoe$ in 
the weak-$*$ topology $\sigma(\ztp,\zt)$,  up to a subsequence. 
Moreover, using the Lyapunov
functional $\cH[\cdot]$,  one has also that $\rhoe^\delta-\rhoe  \sim o_\delta(1)$ in $Y_T$.
Now as $\dt \rhoe^\delta$ is regular enough, one derives the ODE solved by $\cke$,
\begin{equation}\label{eq.cked}
\left\{
 \begin{aligned}
 & \da \cked + \zteps^\delta(\bfx,a,t+\e a) \cked = \dt \rhoe^\delta - (a D^{\e a}_t \zteps^\delta) \rhoe^\delta,& \text{ a.e. } (\bfx,a,t) \in P_\e,\\
 & \cked(\bfx,0,t) = 0,& \text{ a.e. }\bfx \in \Omega,\; a=0, \; \text{ a.e. }t>0.
\end{aligned}
\right.
\end{equation}
This can be integrated and gives :
$$
\begin{aligned}
&  \cked(\bfx, a,t ) :=\chiu{P_\e} (\bfx,a,t) 
 \times  \int_0^a \exp\left( - \int_\tia^a \zteps^\delta(\bfx,s,t+\e s) ds \right) \left\{ \dt \rhoe^\delta(\bfx,\tia,t) - (\tia D^{\e \tia}_t \zteps^\delta) \rhoe^\delta(\bfx,\tia,t) \right\} d\tia.
\end{aligned} 
$$
Tested against $\psi \in \zt$ and integrated on $\cqt$, this becomes :
$$
\begin{aligned}
&  \int_{\cqt} \psi(\bfx,a,t) \cked (\bfx,a,t) d \bfx da dt =  -  \int_\cqt \int_0^a \exp\left( - \int_\tia^a \zteps^\delta(s,t+\e s) ds \right) (\tia D^{\e \tia}_t \zteps^\delta) \rhoe^\delta(\tia,t)  d\tia da dt d \bfx  + \\
&+  \int_\om \int_0^T \int_0^{\frac{T-t}{\e}} \int_{\tia}^{\frac{T-t}{\e}}  \exp\left( - \int_\tia^a \zteps^\delta(\bfx,s,t+\e s) ds \right) \psi(\bfx,a,t) da \dt \rhoe^\delta(\bfx, \tia ,t) d\tia dt d \bfx.
\end{aligned}
$$
Setting 
$$
\varphi_\e^\delta(\bfx,\tia,t):=\chiu{P_\e} (\bfx,\tia,t)  \int_{\tia}^{\frac{T-t}{\e}}  \exp\left( - \int_\tia^a \zteps^\delta(\bfx,s,t+\e s) ds \right) \psi(\bfx,a,t) da, $$
one recovers the regularized version of \eqref{eq.cke}.
Since $\zteps^\delta \in W^{1,\infty}(\cqt)$, one has $\zteps^\delta \to \zteps$ strongly in $C(K)$ for any compact $K \subset \cqt$.
Thus one has : $\varphi_\e^\delta \to \varphi_\e$ strongly  in $L^\infty_{a,t}L^1_\bfx$ 
when $\delta \to 0$. Since $\varphi_\e^\delta$ and $\varphi_\e$ are continuous and compactly supported, the strong
convergence occurs as well in $Z_T$.
There exists a subsequence ${\dt \rhoe^\delta}$
converging in the $\sigma(\ztp,\zt)$ topology to ${\dt \rhoe}$ thus
$$
\begin{aligned}
 \int_{\cqt} &\varphi_\e^\delta \dt \rhoe^\delta d\tia dt d\bfx 
  - \left\langle \dt \rhoe , \varphi_\e \right\rangle_\ztpair{T}  \\
& = \int_{\cqt} (\varphi_\e^\delta -\varphi_\e) \dt \rhoe^\delta d\tia dt d\bfx  
+ \int_{\cqt} \varphi_\e  \dt \rhoe^\delta  d\tia dt d\bfx - \langle  \dt \rhoe , \varphi_\e \rangle _\ztpair{T} \to 0,
 \end{aligned}
$$
when $\delta$ goes to zero. 
Now other arguments using the strong convergence of $\rhoe^\delta$ justify the claim. 
Moreover, $\varphi_\e$ is bounded uniformly with respect to  $\e$. Indeed :
$$
\nrm{\varphi_\e(\cdot,\tia,t)}{L^1(\om)} \leq \chiu{P_\e} \int_{\tia}^{\frac{T-t}{\e}} \expm{(a-\tia)} \nrm{\psi(\cdot,a,t)}{L^1(\om)} da,
$$
which gives after taking the sup over $\oqt$ that 
\begin{equation}\label{eq.unif.bound.phi.eps}
\nrm{\varphi_\e}{Z_T}\leq \sup_{(\tia,t)\in P_\e}\int_0^{\frac{T-t}{\e} -\tia} \expm{a} da \nrm{\psi}{Z_T} \lesssim \nrm{\psi}{Z_T}. 
\end{equation}
\end{proof}

\begin{corollary}\label{coro.cvg.phi}
 Under the previous hypotheses, one has  that $\nrm{\varphi_\e(\cdot,a,t) - \varphi_0(\cdot,a,t)}{L^1_\bfx}$ tends to zero when $\e$ goes to zero, for every 
 fixed $(a,t) \in \qt$, where 
 \begin{equation}\label{eq.def.varphi.zero}
 	\varphi_0 (\bfx,\tia,t) :=  \int_\tia^\infty \psi(\bfx,a,t) \exp\left( - \int_\tia^a \ztz (\bfx,s,t) ds \right) da.
 \end{equation}
\end{corollary}

\begin{proposition}\label{prop.first.term.right hand side.cke}
 Under hypotheses \ref{hypo.data} and \ref{hypo.data.deux}, one has
 $$
 \langle  \dt \rhoe ,\varphi_\e \rangle  \to \int_{\cqt} \dt \rhoz \varphi_0 d\bfx da dt - \int_{\om\times\rr} \varphi_0(\bfx,a,0) (\rhoi(\bfx,a)-\rhoz(\bfx,a,0)) d\bfx da,
 $$
 where $\varphi_\e$ is defined in \eqref{eq.phi.eps}.
\end{proposition}
\begin{proof}
We set $\ell :=  \int_{\om\times\rr} \varphi_0(\bfx,a,0) (\rhoi(\bfx,a)-\rhoz(\bfx,a,0)) d\bfx da$.
 As above one has
 $$
\begin{aligned}
 & \left|  \langle  \dt \rhoe ,\varphi_\e \rangle  - \langle \dt \rhoz ,\varphi_0 \rangle  +\ell \right|
 \leq \left|  \langle  \dt \rhoe -\dt \rhoz - \dt \trhoze,\varphi_\e \rangle  \right|
+\left|  \langle  \dt \rhoz ,\varphi_\e -\varphi_0 \rangle  \right| + \\
&+ \left|  \langle  \dt \trhoz ,\varphi_\e  \rangle  +\ell \right| \leq o_\e(1) + \left|\langle  \dt \trhoze,\varphi_\e\rangle + \ell\right| 
=: o_\e(1) + J_\e,
\end{aligned}
$$
the first term in the right hand side is $o_\e(1)$ thanks to Proposition \ref{prop.err.est}. We focus on the second one~: 
thanks to Corollary \ref{coro.cvg.phi} and  as $\esupxt | \dt \rhoz (\bfx,a,t)|$ is an integrable function on $\qt$,
$$
\left| \int_\cqt (\varphi_\e-\varphi_0) \dt \rhoz d\bfx da dt \right| \leq 
\int_{\qt} \nrm{\left(\varphi_\e-\varphi_0\right)(\cdot,a,t)}{L^1(\om)} \esupx \left| \dt \rhoz (\bfx,a,t) \right| da dt.
$$
By Lebesgue's Theorem, the right hand side tends to zero.
Using Corollary \ref{coro.scale}, one writes then
$$
\begin{aligned}
J_\e  = &\left| \langle  \dt \trhoz,\varphi_\e(\cdot,\cdot,\e\cdot)\rangle _\ztpair{\Tse}  \right| 
\leq  \left| \langle  \dt \trhoz,(\varphi_\e-\varphi_0) (\cdot,\cdot,\e\cdot)\rangle _\ztpair{\Tse}  \right|  \\
& + \left| \langle  \dt \trhoz,\varphi_0(\cdot,\cdot,\e\cdot)-\varphi_0 (\cdot,\cdot,0)\rangle _\ztpair{\Tse}  \right| 
 + \left| \langle  \dt \trhoz,\varphi_0 (\cdot,\cdot,0)\rangle _\ztpair{\Tse}   +\ell\right| = \sum_{i=1}^3 J_{\e,i}.
\end{aligned}
$$
 Since $\chiu{(0,\Tse)}(t)\nrm{(\varphi_\e-\varphi_0)(\cdot,a,\e t)}{L^1(\om)}$ and
 $\chiu{(0,\Tse)}(t)\nrm{\varphi_0(\cdot,a,\e t)-\varphi_0 (\cdot,a,0)}{L^1(\om)}$ tend to zero 
 for a.e. $(a,t)\in(\rr)^2$, by the same arguments as in the proof of Proposition \ref{prop.cvg.diff},
 one concludes that $J_{\e,1}$ and $J_{\e,2}$ vanish when $\e$ goes to 0.
For the last term we use 
 Theorem \ref{thm.cvg.il.micro}, and
 one concludes.
 \end{proof}

Now we are in the position to prove
\begin{proposition}\label{prop.lim.cke}
 Under hypotheses \ref{hypo.data} and \ref{hypo.data.deux}, when $\e$ goes to zero,
 $$
\begin{aligned}
 \int_\cqt \cke(\bfx, a,t) & \psi (\bfx, a,t)da dt d\bfx  \to \int_\cqt  a \dt \rhoz(\bfx,a,t)\psi(\bfx,a,t) da dt d\bfx  - \\
 & - \int_{\om\times\rr} \varphi_0(\bfx,a,0) (\rhoi(\bfx,a)-\rhoz(\bfx,a,0)) da d\bfx
\end{aligned}
$$
for all $\psi \in C_b(\ocyt)$ and $\varphi_0$ is defined in \eqref{eq.def.varphi.zero} and depends on $\psi$.
\end{proposition}
 
\begin{proof}
 Considering the first term in \eqref{eq.cke}, Proposition \ref{prop.first.term.right hand side.cke}, shows that~:
  $$
 \langle  \dt \rhoe ,\varphi_\e \rangle  \to \int_{\cqt} \dt \rhoz \varphi_0 d\bfx da dt - \int_{\om\times\rr} \varphi_0(\bfx,a,0) (\rhoi(\bfx,a)-\rhoz(\bfx,a,0)) d\bfx da.
 $$
On the other hand, hypotheses \ref{hypo.data}, standard arguments and the strong convergence of $\rhoe$ imply that
$$
\int_\cqt \varphi_\e   a ( D^{\e a}_t \zteps) \rhoe(\bfx,a,t) da dt d\bfx \to 
\int_\cqt \varphi_0 (\bfx,a,t)   a \dt \ztz (\bfx,a,t)  \rhoz(\bfx,a,t) da dt d\bfx.
$$
So that finally, one has 
$$
\begin{aligned}
\lim_{\e \to 0}  & \int_\cqt \cke(\bfx, a,t) \psi(\bfx,a,t) d\bfx da dt = \int_\cqt (\dt \rhoz - a \dt \ztz \rhoz ) \varphi_0 d\bfx da dt - \\
& -  \int_{\om\times\rr} \varphi_0(\bfx,a,0) (\rhoi(\bfx,a)-\rhoz(\bfx,a,0)) d\bfx da.
\end{aligned}
$$
As $\cK_0(\bfx,a,t) :=a \dt \rhoz(\bfx,a,t)$ is solving 
$$
(\da +\ztz(\bfx,a,t))\cK_0 = \dt \rhoz - a \dt \ztz \rhoz, \quad \cK_0(\bfx,0,t)=0,
$$
it is explicit and reads :
$$
\cK_0(\bfx,a,t) = \int_0^a \exp\left( -\int_\tia^a \ztz(\bfx,s,t)ds \right) \left( \dt \rhoz - \tia \dt \ztz \rhoz \right) d\tia.
$$
It is then a matter of check to write  
$$
\int_\cqt \cK_0 \psi da dt \bfx = \int_\cqt \varphi_0 \left( \dt \rhoz - a \dt \ztz \rhoz \right) da dt d\bfx,
$$
which ends the proof.
\end{proof}
\newcommand{\bpsi}{{\boldsymbol \psi}}

\subsection{Convergence of $\zeps$ } 
In \cite{MiOel.4}, we derived, uniformly with respect to $\e$,  $L^\infty_tL^1_\bfx$ estimates for $\dt \zeps$.
Here we were not able to obtain this uniformity with respect to $\e$,  and numerical
simulations showed  that these estimate do not hold true here. 
Thus the rest of the paper deals with the asymptotic when $\e$
goes to zero when only $\bL^2_{\bfx,t}$  compactness for $\zeps$  is available.

%

We consider 
$$
\caIe (\rhoe,\zeps,\bpsi) := \ue \int_{\Omega} \int_0^T \int_\rr \rhoe(\bfx,a,t) \left( \zeps(\bfx,t) - \zeps(\bfx,t-\e a) \right) da \bpsi(\bfx,t) dt d \bfx
$$
and we want to express the limit of this operator when $\e$ goes to 0. 
\begin{theorem}\label{thm.cvg.weak}
 Under hypotheses \ref{hypo.data}, \ref{hypo.data.deux} and \ref{hypo.data.trois}, when $\e$  goes to zero, one has that
 $$
\begin{aligned}
 \caIe (\rhoe,\zeps,\bpsi) \to &  \left[ \int_{\Omega} \bpsi(\bfx,t) \cdot \zz(\bfx,t) \muoz(\bfx,t) d\bfx \right]_{t=0}^{t=T}  - \int_{\om\times(0,T)} \zz(\bfx,t) \cdot \dt \left( \muoz \bpsi \right) dt d\bfx,
\end{aligned}
 $$
 where $\muoz(\bfx,t) := \int_{\rr} a \rhoz(\bfx,a,t) da$ and for any test function $\bpsi \in W^{1,\infty}(\qt)$.
\end{theorem}

\begin{proof}
 Splitting the domain of integration, 
$$
\begin{aligned}
&\caIe (\rhoe,\zeps,\bpsi)  =  \ue \int_{\Omega} \int_0^T \int_{\frac{T-t}{\e}}^{\infty} \rhoe(\bfx,a,t) \zeps(\bfx,t) \cdot \bpsi(\bfx,t) da dt d\bfx \\
&- \ue \int_{\Omega} \int_0^T \int_{0}^{\frac{T-t}{\e}} \zeps(\bfx,t)\cdot \bpsi(\bfx,t)  \left( \rhoe(\bfx,a,t+\e a) - \rhoe(\bfx,a,t) \right) da dt d\bfx
\\
& - \ue \int_{\Omega} \int_0^T \int_{0}^{\frac{T-t}{\e}} \zeps(\bfx,t)\cdot \left( \bpsi(\bfx,t+\e a) - \bpsi(\bfx,t) \right) \rhoe(\bfx,a,t+\e a) da dtd\bfx
\\& - \ue \int_{\Omega} \int_0^T \int_{\frac{t}{\e}}^{\infty} \rhoe(\bfx,a,t) \zeps(\bfx,t-\e a) \cdot \bpsi(\bfx,t) da dt d\bfx =: \ell_1 -\sum_{i \in \{2,3,4\}} \ell_i\,.
\end{aligned}
$$
Thanks to \ref{sec.weak.form.initial}, 
$\ell_1$ and $\ell_4$ tend respectively to 
$$
\ell_1 \to  \int_{\Omega} \zz(\bfx,T) \cdot \bpsi(\bfx,T) \muoz(\bfx,T) d\bfx, \quad \ell_4 \to  \int_{\Omega} \zz(\bfx,0) \cdot
 \int_{\rr} \bvarphi_0(\bfx,a,0)  \rhoi(\bfx,a) da d\bfx,
$$
where $\bvarphi_0(\bfx,a,t) := \bpsi(\bfx,t)\int_a^\infty \exp\left( - \int_a^\tia \ztz(\bfx,s,t) ds \right)  d\tia$.
By strong convergence established for $\rhoe$ and $\zeps$, and the Lebesgue's Theorem, one  shows that
$$
\ell_3 \to \int_{\Omega \times \rr \times (0,T)} \zz(\bfx,t) \cdot \dt \bpsi(\bfx,t)  a  \rhoz(\bfx,a,t)dt da d \bfx.
$$
Setting $\hz = \zeps - \zz$, and rewriting $\ell_2$ gives :
$$
\begin{aligned}
\ell_2 = & \int_\cyt  \cke(\bfx,a,t) \zeps(\bfx,t)\cdot \bpsi(\bfx,t) da dt d\bfx \\
&= \int_\cyt \hz(\bfx,t) \cdot\bpsi(\bfx,t) \cke(\bfx,a,t)  da dt d\bfx  +  \int_\cyt  \cke(\bfx,a,t) \zz (\bfx,t)\cdot\bpsi(\bfx,t) da dt d\bfx . 
\end{aligned}
 $$
Thanks to Corollary \ref{coro.conv.zeps} and Proposition \ref{prop.lim.cke}, one concludes that~:
$$
\begin{aligned}
 \ell_2   \to & \int_{\cqt} a \dt \rhoz(\bfx,a,t) \zz(\bfx,t)  \cdot \bpsi(\bfx,t) dt da d \bfx 
 - \int_{\rr} \int_{\Omega}  (\rhoi(\bfx,a)- \rhoz(\bfx,a,0)) \zz(\bfx,0)\cdot \bvarphi_0(\bfx,a,0) da d\bfx 
\end{aligned}
$$
then gathering the terms provides that 
$$
\begin{aligned}
\cIe (\rhoe,\zeps,\bpsi)  & \to \int_{\Omega\times \rr} a \rhoz(\bfx,a,T) \zz(\bfx,T) \cdot \bpsi(\bfx,T) da d\bfx 
-\int_{\cqt} a(  \rhoz \zz \cdot \dt \bpsi +  \dt\rhoz \zz \cdot \bpsi ) dt da d \bfx \\
&-  \int_{\Omega\times\rr} \rhoz(\bfx,a,0)\bvarphi_0(\bfx,a,0) da \cdot \zz(\bfx,0) d\bfx . 
\end{aligned}
$$
where $\bvarphi_0$ is defined in \eqref{eq.def.varphi.zero} as a function of $\bpsi$.
The last term of the previous right hand side can then be transformed into  :
$$
\begin{aligned}
 \int_{\Omega} \int_{\rr} \int_0^\tia & \rhoz(\bfx,a,0)\exp\left( - \int_a^\tia \ztz(\bfx,s,0)ds \right) da d \tia \bpsi(\bfx,0) \cdot \zz(\bfx,0) d\bfx \\
 & = \int_{\Omega} \left\{ \int_\rr a \rhoz(\bfx,a,0) da \right\} \bpsi(\bfx,0) \cdot \zz(\bfx,0) d\bfx = \int_{\Omega} \bpsi(\bfx,0)\cdot \zz(\bfx,0) \muoz(\bfx,0)   d\bfx,
\end{aligned}
$$
since $\da (a \rhoz) + \ztz (a \rhoz ) = \rhoz$, one sees easily that the latter inner integral corresponds exactly to the
integration of the latter ODE.
\end{proof}

%

\begin{theorem}
 Under hypotheses \ref{hypo.data}, \ref{hypo.data.deux} and \ref{hypo.data.trois}, the unique solution $\zeps$ of the problem \ref{eq.eul.lag.cont}
 converges towards the unique solution pair $\zz \in H^1((0,T);\bL^2(\om))\cap L^\infty((0,T);\bho)$  satisfying~:
 $$
\begin{aligned}
&     \int_{\om\times(0,T)} \muoz \dt \zz \cdot \bvarphi+ \dx \zz \cdot \dx \bvarphi -   |\dx \zz|^2 \zz \cdot \bvarphi \; \,d\bfx\,dt   =0
\end{aligned}
 $$
 for every $\varphi \in  H^1((0,T);\bL^2(\om))\cap L^\infty((0,T);\bho)$.
\end{theorem}
\begin{proof} From Theorem \ref{thm.cvg.weak}, and stability results above one has that the solution $\zeps$ 
satisfying the weak formulation \eqref{eq.weak.z.eps} converges strongly in $C(\ov{\om}\times[0,T])$ to $\zz\in H^1((0,T);\bL^2(\om))\cap L^\infty((0,T);\bho)$ satisfying :
  $$
\begin{aligned}
&    \left[ \int_{\Omega} \psi(\bfx,t) \muoz(\bfx,t) \zz(\bfx,t) d\bfx \right]_{t=0}^{t=T}  - \int_{\qt} \zz(\bfx,t) \dt \left( \muoz \psi \right) dt d\bfx  \\
&+ \int_{\om\times(0,T)} \dx \zz \cdot \dx \varphi d\bfx\,dt  + \langle \zz \cdot \varphi , \lamz \rangle =0
\end{aligned}
 $$
 together with the constraint $| \zz |= 1$ everywhere in $\ov{\om}\times[0,T]$ obtained in the same way as in Theorem \ref{thm.lim.delta}.
Then since in the first line of the latter equation all terms are well defined, one can perform an integration 
by parts in time and obtain that $\zz$ solves~:
 $$
\begin{aligned}
&     \int_{\om\times(0,T)} \muoz\dt \zz \cdot \bvarphi+ \dx \zz \cdot \dx \bvarphi \; \,d\bfx\,dt +  \langle \zz \cdot \bvarphi ,\lamz \rangle   =0
\end{aligned}
 $$
 again as in Theorem \ref{thm.lim.delta}, choosing the test function to be $\bvarphi= \zz \theta$, with $\theta \in \D(\om\times(0,T))$,  one proves that
 for almost every $(\bfx,t)$, $\lamz = - | \dx \zz |^2$ which, because $\zz$ belongs to  $L^\infty_t \bH^1_\bfx$, is an $L^1_{\bfx,t}$ function.
 Uniqueness follows the same ideas as in the proof of Theorem \ref{thm.lim.delta}, it is simpler since the first term is a derivative instead 
 of being a delay term as in \eqref{eq.eul.lag.cont}. Indeed, 
 denoting $\hz:=\zz^2-\zz^1$, where $(\zz^i)_{i\in\{1,2\}}$ are two distinct solutions, it solves for every $\bfv \in \bho$ and for almost every $t\in(0,T)$, 
 $$
\begin{aligned}
&     \int_{\om} \muoz \dt \hz \cdot \bfv+ \dx \hz \cdot \dx \bfv  + \nud \left\{ (\lamz^2+\lamz^1) \hz + \hat{\lambda}_0 ( \zz^2+\zz^1 ) \right\} \cdot \bfv \; \,d\bfx\,dt   =0
\end{aligned}
 $$
choosing  $\bfv=\hz$ and using Gagliardo-Nirenberg as above, one recovers :
$$
\frac{\mumin}{2} \ddt{}  \nrm{\hz}{L^2(\om)}^2   + (1-\delta C ) \nrm{\dx \hz}{L^2(\om)}^2 \leq \frac{C}{\delta} \nrm{\hz}{L^2(\om)}^2.
$$
For $\delta$ small enough, one neglects the second term in the left hand side. Thanks to Gronwall's Lemma, one concludes, 
since $\hz(\bfx,0)=0$ for all $x\in\om$,
that $\nrm{\hz(\cdot,t)}{L^2(\om)}^2=0$ a.e. $t \in (0,T)$ which  shows the claim.
\end{proof}
%


\appendix

\section{Initial and final  terms in the weak formulation}\label{sec.weak.form.initial}
\begin{proposition}
 Under hypotheses \ref{hypo.data} and \ref{hypo.data.deux}, one has 
$$  \ell_4 = \ue \int_{\Omega} \int_0^T \int_{\frac{t}{\e}}^{\infty} \rhoe(\bfx,a,t) \zeps(\bfx,t-\e a) \cdot \bpsi(\bfx,t) da dt d\bfx 
\to  \int_{\Omega} \zz(\bfx,0) 
\cdot
 \int_{\rr} \bvarphi_0(\bfx,a,0)  \rhoi(\bfx,a) da d\bfx,
$$
where $\ell_4$ and $\bvarphi_0$ are defined in the proof of Theorem \ref{thm.cvg.weak}
\end{proposition}
\begin{proof}
 Using the characteristics, one writes :
$$
\begin{aligned}
\ell_4 = &\ue \int_{\Omega} \int_0^T \int_{\frac{t}{\e}}^{\infty} 
\rhoi(\bfx,a-t/\e) \exp\left( - \int_0^{\frac{t}{\e}} \zteps\left(\bfx,a+s -\tse,\e s\right)ds \right) \times \\
& \times \zp(\bfx,t-\e a) \bpsi(\bfx,t) da dt d\bfx \\
& = \ue \int_{\Omega} \int_0^T \int_{0}^{\infty} 
\rhoi(\bfx,\tia) \exp\left( - \int_0^{\frac{t}{\e}} \zteps\left(\bfx,\tia+s,\e s\right)ds \right) \zp(\bfx,-\e \tia ) \bpsi(\bfx,t) d\tia dt d\bfx \\
& =  \int_{\Omega}  \int_0^{T/\e} \int_{0}^{\infty} 
\rhoi(\bfx,\tia) \exp\left( - \int_0^{\tit } \zteps\left(\bfx,\tia+s,\e s\right)ds \right) \zp(\bfx,-\e \tia ) \bpsi(\bfx,\e \tit) d\tia d\tit d\bfx.
\end{aligned}
$$
By the Lebesgue's Theorem the latter term tends to 
$$
\begin{aligned}\lim_{\e \to 0} \ell_4 & = \int_{\Omega} \zp(\bfx, 0) \bpsi(\bfx,0) \int_\rr \int_\rr \rhoi(\bfx, a) \left( - \int_0^{t } \ztz\left(\bfx, a+s,0 \right)ds \right)  d a dt d\bfx\\
& =   \int_{\Omega} \zp(\bfx, 0) \int_{\rr} \rhoi(\bfx,a) \bvarphi_0(\bfx,a,0) da d \bfx,
\end{aligned}
$$
where $\bvarphi_0 (a,t)$ is defined as
\end{proof}

\begin{proposition}
 Under hypotheses \ref{hypo.data} and \ref{hypo.data.deux}, one has 
$$  \ell_1  \to  \int_{\Omega} \zz(\bfx,T) \bpsi(\bfx,T) \muoz(\bfx,T) d\bfx, 
$$
where $\ell_4$ and $\bvarphi_0$ are defined in the proof of Theorem \ref{thm.cvg.weak}.
\end{proposition}

\begin{proof}
\begin{figure}[ht!]
\begin{center}
 \begin{tikzpicture}[scale=1.5]
\draw [->] (0,0) -- (2.9,0);
\draw (2.9,0) node[right] {$a$};
\draw [->] (0,0) -- (0,2.25);
\draw (0,2.25) node[above] {$t$};
\draw (0,2) node[left] {$T$};
\draw (0,1) node[left] {$\frac{T}{2}$};
\draw (0,1) node  {--};
\draw [-] (0,2) -- (2.9,2);
\draw (2.5,0) node[below] {$\frac{T}{\varepsilon}$};
\draw [-] (0,2) -- (2.5,0);
\draw [-] (2.5,0) -- (2.5,2);
\draw [dashed] (0,0) -- (2.5,2);
\draw [-] (2,0.4) -- (2,2);
\draw (2,0.4) node[below] {\tiny $(a,T-\varepsilon a)$};
\draw (2,8/5) node[below] {\tiny $(a,\varepsilon a)$};
\draw (2,2) node[above] {\tiny $(a,T)$};
\draw (1.25,0) node [below] {$\frac{T}{2 \varepsilon}$};
\draw (1.25,0) node {$|$};
\draw [dotted] (1.25,0) -- (1.25,2);
\end{tikzpicture} 
\end{center}
\caption{Partition of the age-time domain for the splitting of the $\ell_{1}$ terms}\label{fig.age.time}
\end{figure}
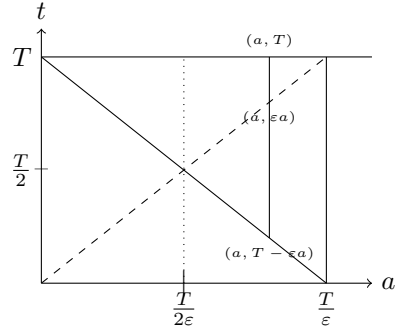
Dividing $(0,T)$ in two equal parts $(0,T/2)$ and $(T/2,T)$,  gives two terms
$\ell_{1,1}$ and $\ell_{1,2}$.
For the first one, we have
$$
\begin{aligned}
\ell_{1,1}= & \ue \int_{\om\times\left(0,\frac{T}{2}\right)} \int_{\frac{T-t}{\e}}^\infty \rhoi(\bfx,a-t/\e) \exp\left( -\int_0^\tse \zteps(\bfx,a-t/\e+s,\e s)ds \right) \times \\
&\times \zeps(\bfx,t)\bpsi(\bfx,t) da dt d\bfx ,\\
& = \frac{1}{\e} \int_{\om\times\left(0,\frac{T}{2}\right)}   \int_{\frac{T-2 t}{\e}}^\infty \rhoi(\bfx,a) \exp\left( -\int_0^\tse \zteps(\bfx,a+s,\e s)ds \right) \zeps(\bfx,t)\bpsi(\bfx,t) da dt\, d \bfx,\\
&= \int_{\om\times\left(0,\frac{T}{2\e}\right)}  \int_{\frac{T}{\e}-2 t}^\infty \rhoi(\bfx,a) \exp\left( -\int_0^t \zteps(\bfx,a+s,\e s)ds \right) \zeps(\bfx,\e t)\bpsi(\bfx,\e t) da dt d\bfx.
\end{aligned}
$$
Now, we estimate the latter term as
$$
\begin{aligned}
 \left| \ell_{1,1}  \right| & \leq C \int_\om  \int_0^{\frac{T}{2\e}} \int_{\frac{T}{\e}-2 t}^\infty \exp(-\ztmin t) \rhoi(\bfx,a) da dt d\bfx \\
 & =  \int_\om \int_0^{\frac{T}{2\e}} \int_{\frac{T}{\e}-2 t}^{\frac{T}{\e}}  \exp(-\ztmin t) \rhoi(\bfx,a) da dt d\bfx \\
 & + \int_\om \int_0^{\frac{T}{2\e}}  \int_{\frac{T}{\e}}^\infty \exp(-\ztmin t) \rhoi(\bfx, a) da dt d \bfx = \ell_{1,1,1}+\ell_{1,1,2},
\end{aligned}
$$ 
applying Lebesgue's Theorem shows that $\ell_{1,1,2} \to 0$  when $\e$ goes to zero.
$$
\begin{aligned}
\ell_{1,1,1} & =  \int_\om \int_0^{\frac{T}{\e}} \rhoi(\bfx,a) \int_{\frac{T}{2\e}-\frac{a}{2}}^{\frac{T}{2\e}} \exp(-\ztmin t)dt da d\bfx \\
& \lesssim  \int_\om \int_0^{\frac{T}{\e}} \rhoi(\bfx, a) 
\left( 
\exp\left(-\ztmin \left(\frac{T}{2 \e}-\frac{a}{2}\right)\right) 
- 
\exp\left(-\frac{\ztmin T}{2 \e}\right)
\right)
 da d\bfx, \\
& \lesssim  
\exp\left(-\frac{\ztmin T}{2 \e}\right) 
\int_0^{\Tse} 
\nrm{\rhoi(\cdot,a)}{L^\infty_\bfx} (1+a) da 
\sup_{a \in \left( 0,\Tse\right) } 
\frac{\left(\exp\left(\frac{\ztmin a}{2}\right)-1\right)}{1+a} \\
& \lesssim \frac{\e}{T} \nrm{(1+a)\rhoi}{L^1_a L^\infty_\bfx}.
\end{aligned}
$$
where in the third line, the function $(\exp(\ztmin a/2)-1)/(1+a)$ is monotone increasing provided that $\ztmin>0$,
which then gives the result.
The term $\ell_{1,2}$ can be split in two parts~:
$$
\begin{aligned}
\ell_{1,2} := & \ue \int_\om \int_{\frac{T}{2}}^T \int_{\frac{T-t}{\e}}^\tse \rhoe(\bfx,a,t) \zeps(\bfx,t) \bpsi(\bfx,t) da dt d \bfx \\
 &+ \ue \int_\om \int_{\frac{T}{2}}^T \int_\tse^\infty \rhoe(\bfx,a,t) \zeps(\bfx,t) \bpsi(\bfx,t) da dt d \bfx=\ell_{1,2,1}+ \ell_{1,2,2},
\end{aligned}
$$
Using Duhamel's principle again for $t<\e a$ provides :
$$
\begin{aligned}
\left| \ell_{1,2,2} \right| = & \left|  \int_\om \int_{\frac{T}{2 \e}}^{\frac{T}{\e}} \int_\rr \rhoi(\bfx,a) \exp\left( - \int_0^t \zteps(\bfx,a+s,\e s) ds \right) \,  da\, \zeps(\bfx,\e t)\bpsi(\bfx,\e t) dt d \bfx \right| \\
&  \lesssim \exp\left(-{\frac{T}{2\e}}\right) \nrm{\rhoi}{L^1(\rr;L^\infty(\Omega))} \to 0.
\end{aligned}
$$ 
Finally,   we split $\ell_{1,2,1}$ in two terms
$$
\ell_{1,2,1} = \ue \int_\om \left\{ \left( \int_{\frac{T}{2}}^T \int_{\frac{T-t}{\e}}^{\frac{T}{2\e}} 
+ \int_{\frac{T}{2}}^T \int_{\frac{T}{2\e}}^\tse \right) \rhoe(\bfx,a,t) \zeps(\bfx,t) \bpsi(\bfx,t) da dt \right\} d\bfx = \ell_{1,2,1,1} + \ell_{1,2,1,2}. 
$$
In a straightforward manner, the latter term goes to zero since 
$$
\begin{aligned}
\ell_{1,2,1,2}:= & \ue \int_\om \int_{\frac{T}{2}}^T \int_{\frac{T}{2\e}}^\tse  \rhoe(\bfx,a,t) \zeps(\bfx,t) \bpsi(\bfx,t) da dt d\bfx \\
& =  \ue \int_\om \int_{\frac{T}{2}}^T \int_{\frac{T}{2\e}}^\tse  \rhoe(\bfx,0,t-\e a) \exp\left(-\int_{-a}^0 \zteps(\bfx,a+s,t+\e s) ds\right) \zeps(\bfx,t) \bpsi(\bfx,t) da dt d\bfx \\
&  \leq \frac{\bmax}{\e} |\om| \nrm{ \zeps \bpsi}{L^\infty_{\bfx,t}}\int_{\frac{T}{2}}^T \int_{\frac{T}{2\e}}^\tse    \exp(-\ztmin a) da dt  \lesssim \exp(-\ztmin T/(2 \e))/\e \sim o_\e(1)
\end{aligned}
$$
where we used that $\rhoe(\bfx,0,t-\e a)$ is bounded uniformly with respect to  $\e$.
The only term remaining and which should converge to a non-zero limit is 
$$
\begin{aligned}
&J:=\ell_{1,2,1,1}:=  \ue  \int_\om \int_{\frac{T}{2}}^T \int_{\frac{T-t}{\e}}^{\frac{T}{2\e}}  \rhoe(\bfx,a,t) \zeps(\bfx,t) \bpsi(\bfx,t) da dt d \bfx\\
& = \ue \int_\om \int_0^{\frac{T}{2\e}} \int_{T-\e a}^T \rhoe(\bfx,0,t-\e a) \exp\left( -\int_0^a \zteps(\bfx,s,t+\e (s-a) )ds \right) \zeps(\bfx,t) \bpsi(\bfx,t) dt da d\bfx.
\end{aligned}
$$
We denote by $\omega_\e := \Omega\times\{ (a,t) \in \rr\times(0,T) \st t \in (T/2,T) \text{ and } a \in ((T-t)/\e,T/\e)\}$.
$$
J = \ue \left\{ \int_{\omega_\e} \hrhoe \zeps \bpsi dq +  \int_{\omega_\e} \rhoz \hz \bpsi dq +  \int_{\omega_\e} \rhoz \zz \bpsi dq \right\} = \sum_{i=1}^3 J_i,
$$
where $\hrhoe=\rhoe-\rhoz$.
Thanks to the Duhamel's principle applied to $\hrhoe$, one has the estimate 
$
|J_1| \sim o_\e(1)
$.
Indeed as in Lemma 5.2 \cite{MiOel.4}, when $\e a <t $ one can show that $|\hrhoe(a,t)|\lesssim \exp( -\ztmin t / \e) + (1+a)^2 \exp(-\ztmin a)o_\e(1)$, which gives the result.
By Lemma 3.4 \cite{MiOel.4}, 
$| \rhoz(\bfx,a,t) | \lesssim \exp(-\ztmin a)$ and $\nrm{\hz}{L^\infty((0,T);L^2(\Omega))}\sim o_\e(1)$, one has
$$
|J_2| \lesssim\frac{ \nrm{\hz}{ L^\infty((0,T);L^2(\Omega))}}{\e} \int_0^{\frac{T}{2\e}} \int_{T-\e a}^T dt \exp( - \ztmin a) da 
 \sim o_\e (1).
$$
In order to prove the limit of $J_3$ we define $f(a,t) = a \int_\om \rhoz(\bfx,a,t) \bpsi(\bfx,t) \zz(\bfx,t)d\bfx $, and write :
$$
\begin{aligned}
 \left|f(a,t)-f(a,T)\right| \leq  & a \left\{ \exp(-\ztmin a) \left( \nrm{\zz}{L^\infty((0,T);L^2(\Omega))} 
\sup_{\bfx \in \Omega} |\bpsi(\bfx,t)-\bpsi(\bfx,T)| \right. \right. \\
 & \left. \hspace{2cm}+ \nrm{\bpsi}{\bL^\infty(\Omega\times(0,T))} \nrm{\zz(\cdot,t)-\zz(\cdot,T)}{\bL^2(\Omega)} \right)\\
 &\left. +  \nrm{\bpsi}{\bL^\infty(\Omega\times(0,T))} \nrm{\zz}{L^\infty((0,T);\bL^2(\Omega))}\sup_{\bfx \in \Omega} |\rhoz(\bfx,a,t)-\rhoz(\bfx,a,T)| \right\}
\end{aligned}
$$
for all $a \in \rr$ and all $\delta > 0$ there exists $\eta_i$, $i\in\{ 1,2,3\}$ s.t.
$$
\begin{aligned}
\forall t\in (0,T) \st |t-T| < \eta_1 & \implies \sup_{\bfx \in \Omega} | \bpsi(\bfx,t) -\bpsi(\bfx,T) | < \frac{\delta}{3 (1+a) C \exp(-\ztmin a) \nrm{\zz}{L^\infty_t \bL^2_\bfx} }, \\
\forall t \in (0,T) \st   |t-T| < \eta_2 & \implies \nrm{ \zz(\cdot,t) -\zz(\cdot,T)}{L^2(\Omega)} < \frac{\delta}{3 (1+a) C \exp(-\ztmin a) \nrm{\bpsi}{\bL^\infty_{\bfx,t}} }, \\
 \forall t \in (0,T) \st  |t-T| < \eta_3 & \implies  \sup_{\bfx \in \Omega} | \rhoz(\bfx,a,t) -\rhoz(\bfx,a,T) | < 
   \frac{\delta}{3 (1+a) \nrm{\zz}{L^\infty_t\bL^2_\bfx} \nrm{\bpsi}{\bL^\infty_{\bfx,t}} } ,
\end{aligned}
$$
which means that $\forall a\in \rr$, $\forall \delta>0$ there exists $\eta(a,\delta)=\min_{i\in\{1,2,3\}} \eta_i>0$ s.t.
$$
\forall t \in (0,T) \st | t- T | < \eta \implies  \left|f(a,t)-f(a,T)\right| < \delta.
$$
Now there exists $\e(a,\delta)$ small enough $\e < \eta(a,\delta)/(1+a)$ s.t.
for all $t \in(T-\e a,T)$,  $|t-T|< \eta$, thus 
$$
\left| \frac{1}{\e a} \int_{T-\e a}^T  f(a,t)-f(a,T) dt \right| \leq  \frac{1}{\e a} \int_{T-\e a}^T \left| f(a,t)-f(a,T) \right| dt \leq \delta.
$$
Setting $g_\e (a) = \frac{1}{\e a} \int_{T-\e a}^T  f(a,t) dt$ and $g_0(a) = f(a,T)$, this shows that 
there is pointwise convergence for every fixed $a$, $g_\e(a) \to g_0(a)$. Moreover, 
$$
g_\e(a) \lesssim (1+a) \exp(-\ztmin a) \in L^1(\rr)
$$
and by the Lebesgue's Theorem the convergence holds :
$$
\int_\rr \geps(a) da \to \int_\rr a \int_\om \rhoz(\bfx,a,T) \zz(\bfx,T) \cdot \bpsi(\bfx,T) d\bfx da = \int_\om \mu_{1,0}(\bfx,T) \zz(\bfx,T) \cdot \bpsi(\bfx,T) d \bfx 
$$
when $\e$ goes to zero. 
\end{proof}

\section{Functionnal analysis in Banach valued spaces}\label{sec.app.un}


\begin{proof}[of Proposition \ref{prop.cvg.vague}]
 One has :
 $$
 \left|   \int_{\cqt} \did{\sigma}{a} f \varphi(\bfx,a,t) d\bfx da dt  \right|  \leq \nrm{\did{\sigma}{a} f}{Y_T} \nrm{\varphi}{X_T} 
 $$
for any $\varphi \in X_T$. This proves that $\did{\sigma}{a} f$ 
belongs to $X_T'$.  $X_T$ is a separable Banach space. Thus according to Corollary 3.30 p. 76 \cite{Bre.Book}, 
 for any bounded sequence in $X_T'$ there exists a subsequence converging in the weak-$\star$ topology $\sigma(X_T',X_T)$. As the bound $\nrm{\did{\sigma}{a}f}{Y_T}$ is uniform with respect to  $\sigma$, 
there exists a weak limit in $X_T'$ denoted $g$ and we define the $X_T',X_T$ duality pairing as
$$
\lim_{\sigma_k \to 0}\int_{\cqt} \did{\sigma}{a} f \varphi \dxat = \langle g,\varphi\rangle_\xtpair{T},\quad \text{ as } \tau \to 0.
$$
Moreover, $g = \da f$
in the weak sense {\em i.e.} for any $\varphi \in \D(\cqt)$ one has
$$
\int_{\cqt} f (\bfx,a,t) \dt \varphi (\bfx,a,t) d\bfx da dt = - \langle g , \varphi \rangle_\xtpair{T}.
$$
The same results hold obviously for $\dt f$.
\end{proof}


\begin{proof}[of Proposition \ref{prop.cvg.etroite}]
By  Proposition \ref{prop.cvg.vague}, there exists a limit $\da f \in X_T'$ and a subsequence s.t.
$$
\did{\sigma_k}{a} f \wscvg \da f \text{ weak-}\star\text{ in } X_T'.
$$
First, $\did{\sigma}{a}f$ defines a linear continuous functional on $Z_T$ :
\begin{equation}\label{eq.lin.form}
 \left|\int_\cqt \did{\sigma}{a} f \varphi \right| \leq \nrm{\did{\sigma}{a} f}{Y_T} \nrm{\varphi}{Z_T}. 
\end{equation}
Moreover, one has tightness, for any continuous positive $\chim(a)$ s.t. $\chim \equiv1$ on $[0, m]$ and $\supp \, \chim \in [0,m+1]$
 $$
\begin{aligned}
  & \left| \langle  \did{\sigma}{a} f , \varphi (1-\chim) \rangle  \right| \leq \int_{\om \times\rr \times (0,T)} | \did{\sigma}{a} f (\bfx,a,t) \varphi(1-\chim) (\bfx,a,t) | d\bfx d a dt \\
  & \leq \int_0^T \int_\rr \chiu{\supp \varphi(1-\chim)} \esupx | \did{\sigma}{a} f | da dt \nrm{\varphi(1-\chim)}{L^\infty(\qt;L^1(\om))} \\
  & \leq \int_0^T \int_m^\infty \esupx | \did{\sigma}{a} f | da dt \nrm{\varphi(1-\chim)}{X_T} ,
\end{aligned}
 $$
 the latter term goes to zero as $m$ tends to infinity by the Lebesgue's Theorem.
 Now 
 one writes for any $\varphi$ in $Z_T$
 $$ 
\begin{aligned}
&  \left| < \did{\sigma_j}{a}f ,\varphi>-<\did{\sigma_k}{a}f,\varphi> \right| \leq 
  \left| < \did{\sigma_j}{a}f ,\varphi \chim >-<\did{\sigma_k}{a}f,\varphi \chim> \right| \\
  &\quad \quad +   \left| < \did{\sigma_j}{a}f,\varphi (1-\chim) >-<\did{\sigma_k}{a}f,\varphi (1-\chim)> \right|. 
\end{aligned}
$$
For every $\delta>0$ there exists $k_0$ s.t. for $k$ and $j$ greater than $k_0$
the first term on the right hand side is smaller than $\delta/3$ by weak-$\star$ convergence in $\sigma(X_T',X_T)*$,
while  the two latter terms can be made  smaller than $2\delta/3$ due to the tightness proved above.
This implies, because $\RR$ is complete, that there exists a limit $L$ s.t.
$$
L_j:=<\did{\sigma_j}{a}f,\varphi>_\ztpair{T} \to L,\quad \text{when } j \to \infty.
$$
Since for every arbitrary fixed $\delta$ there exists $j_0$ s.t. $j>j_0$ implies 
$$
\left| L \right| \leq \left| L -L_j\right| + |L_j| \leq \delta + C \nrm{\varphi}{Z_T},
$$
$L$ is also a linear continuous form on $\zt$ thanks to \eqref{eq.lin.form}. 
By similar arguments as above we identify this limit with the weak derivative $\da f$ and 
we denote $<\da f,\varphi>_\ztpair{T} := L(\varphi)$ for every $\varphi \in Z_T$. The same proof holds for the time derivative as well.
\end{proof}

\begin{lemma}
  If $f \in U_T$ then
$$
\sup_{
\substack{(\varphi_1,\varphi_2)\in \D(\qt;L^1(\om))\\ |\varphi_i(\xat)|\leq 1,\, i=1,2}} \left| \int_\cqt f (\da \varphi_1 + \dt \varphi_2) d\bfx\; da \; dt  \right|\leq \nrm{f}{Y_T}.
$$
\end{lemma}

\begin{proof}
  Taking the test function $\varphi \in  \D(\qt;L^1(\om))$, one has
$$
\begin{aligned}
\left| \int_\cqt f \dt \varphi \dxat \right| & = 
  \left| \lim_{\sigma \to 0} \int_\cqt f(\xat) \frac{\varphi(\bfx,a,t+\sigma)-\varphi(\xat)}{\sigma} \dxat \right| \\
&  = \left| \lim_{\sigma \to 0} \int_{\supp \varphi+\sigma e_t}  \varphi(\bfx, a, t+\sigma ) \frac{f(\bfx,a,t+\sigma)-f(\xat)}{\sigma} \dxat \right| \\
& \leq \limsup_{\sigma\to 0} \nrm{\did{\sigma}{t}f}{Y_T} \nrm{\varphi}{X_T} \leq \nrm{f}{U_T}.
\end{aligned}
$$
\end{proof}

\begin{lemma}\label{lem.err.bv}
  If $f\in U_T$ then
$$
\begin{aligned}
  \int_0^{T-\Delta t} & \int_\rr \esupx \left|f (\bfx,a+ \Delta a,t+ \Delta t) - \int_0^1 f (\bfx,a+\tau\Delta a,t+\tau \Delta t) d\tau \right| ds \dxat \\
& \leq (\Delta a + \Delta t)\nrm{f}{U_T}.
\end{aligned}
$$
\end{lemma}
\begin{proof}
Thanks to Theorem 2.2.1 in \cite{HeiPatRen.19}, since $f \in U_T = \BV(\rr\times(0,T);L^\infty(\om))$, there exists a smooth 
function $f_\delta \in C^\infty(\rr\times(0,T);L^\infty(\om))$, s.t. for every $\delta$ :
$$
\nrm{f-f^\delta}{L^1(\rr\times(0,T);L^\infty(\om))} \leq \delta ,\quad 
\nrm{\dt f^\delta}{L^1(\rr\times(0,T);L^\infty(\om))} + \nrm{\da f_\delta}{L^1(\rr\times(0,T);L^\infty(\om))} \leq 
\nrm{f}{U_T} + \delta
$$
The first estimates  imply directly that 
$$
 \int_0^{T-\Delta t}  \int_\rr \nrm{f (\cdot,a+ \Delta a,t+ \Delta t) -f^\delta (\cdot,a+ \Delta a,t+ \Delta t)}{L^\infty(\om)} da dt \to 0
$$
as well as
$$
  \int_0^{T-\Delta t}\int_\rr \nrm{\int_0^1 f (\cdot,a+s\Delta a,t+s \Delta t) - f_\delta (\cdot,a+s\Delta a,t+s \Delta t) ds}{L^\infty(\om)}  da dt \to 0
$$
as $\delta \to 0$.
Using that 
$$
f^\delta (\bfx,a+\Delta a,t+\Delta t) = 
f^\delta (\bfx,a,t) + \int_0^1 \nabla_{a,t} f^\delta (\bfx,a+s \Da,t+s \Dt) ds \cdot \begin{pmatrix}
  \Delta a\\ \Delta t
\end{pmatrix}
$$
and
$$
\int_0^1 f^\delta (\bfx,a+s \Delta a,t+s \Delta t) ds = 
f^\delta (\bfx,a,t) + \int_0^1 s \int_0^1 \nabla_{a,t} f^\delta (\bfx,a+s \tau \Da,t+s \tau \Dt) d\tau ds \cdot \begin{pmatrix}
  \Delta a\\ \Delta t
\end{pmatrix}
$$

one then writes that
$$
\begin{aligned}
&f^\delta (\bfx,a+\Delta a,t+\Delta t) -\int_0^1 f^\delta(\bfx,a+s \Delta a,t+s \Delta t) ds \\
& = \left( \int_0^1 \left\{ \nabla_{a,t} f^\delta (\bfx,a+s \Delta a,t+s \Delta t)- s \int_0^1 \nabla_{a,t} f^\delta (\bfx,a+s \tis \Delta a,t+s \tis \Delta t) d\tis \right\}ds \right)\cdot
\begin{pmatrix}
  \Delta a\\ \Delta t
\end{pmatrix}. 
\end{aligned}
$$
Setting $J^\delta(\xat)  := 
f^\delta (\bfx,a+\Delta a,t+\delta t) -\int_0^1 f^\delta(\bfx,a+s \Delta a,t+s \Delta t) ds$ and integrating in age and time gives that
$$
\nrm{ J^\delta }{ Y_T } \leq 2  (\Delta a + \Delta t) \nrm{|\nabla_{a,t} f^\delta|}{Y_T}  \leq  2 (\Delta a + \Delta t) \left( \nrm{f^\delta}{U_T} + \delta\right)  ,
$$
One concludes thanks to a triangular inequality.
\end{proof}

\section*{Acknowledgements}

I would like to thank warmly Ioan Ionescu for his encouragements and comments that helped me in order to finish and submit this work. 
I thank also the referee for the
very attentive reading of this article and for her/his wise comments.

\def\cprime{$'$} \def\cprime{$'$}

%


\end{document}